\documentclass{article}

\usepackage[numbers]{natbib}

\usepackage[preprint]{neurips_2023}

\usepackage[utf8]{inputenc} 
\usepackage[T1]{fontenc}    
\usepackage{hyperref}       
\usepackage{url}            
\usepackage{booktabs}       
\usepackage{amsfonts}       
\usepackage{nicefrac}       
\usepackage{microtype}      
\usepackage{xcolor}         
\usepackage{amsmath}
\usepackage{amssymb}
\usepackage{pifont}
\usepackage{graphicx}
\usepackage{algorithm}
\usepackage{amsthm}
\usepackage{algpseudocode}
\usepackage{cleveref}
\usepackage[flushleft]{threeparttable}
\usepackage{tabularx}
\usepackage{multirow}
\usepackage{colortbl}
\usepackage[font=scriptsize]{caption}
\usepackage{tablefootnote}

\newtheorem{assumption}{Assumption}
\newtheorem{theorem}{Theorem}
\newtheorem{remark}{Remark}
\newtheorem{lemma}{Lemma}
\crefname{assumption}{assumption}{assumptions}
\newtheorem{corollary}{Corollary}

\newcommand{\cmark}{\textcolor{green!80!black}{\ding{51}}}
\newcommand{\xmark}{\textcolor{red}{\ding{55}}}

\definecolor{bgcolor}{rgb}{0.8,1,1}
\definecolor{bgcolor2}{rgb}{0.8,1,0.8}

\title{Optimal Algorithm with Complexity Separation for Strongly Convex-Strongly Concave Composite Saddle Point Problems}
\author{
  Ekaterina Borodich 
  \\MIPT\\ \texttt{borodich.ed@phystech.edu} \\
  \And
   Georgiy Kormakov \\
   Skoltech \\
   \texttt{ Georgiy.Kormakov@skoltech.ru} \\
  \AND
   Dmitry Kovalev \\
   UCL \\
   \texttt{dakovalev1@gmail.com} \\
   \And
   Aleksandr Beznosikov  \\
   MIPT \\
   \texttt{beznosikov.an@phystech.edu} \\
   \And
   Alexander Gasnikov \\
   MIPT \\
   \texttt{gasnikov@yandex.ru} \\
}

\begin{document}

\maketitle
\begin{abstract}
 In this work, we focuses on the following saddle point problem $\min_x \max_y p(x) + R(x,y) - q(y)$ where $R(x,y)$ is $L_R$-smooth, $\mu_x$-strongly convex, $\mu_y$-strongly concave and $p(x), q(y)$ are convex and $L_p, L_q$-smooth respectively. We present a new algorithm with optimal overall complexity $\mathcal{O}\left(\left(\sqrt{\frac{L_p}{\mu_x}} + \frac{L_R}{\sqrt{\mu_x \mu_y}} + \sqrt{\frac{L_q}{\mu_y}}\right)\log \frac{1}{\varepsilon}\right)$ and separation of oracle calls in the composite and saddle part. This algorithm requires $\mathcal{O}\left(\left(\sqrt{\frac{L_p}{\mu_x}} + \sqrt{\frac{L_q}{\mu_y}}\right) \log \frac{1}{\varepsilon}\right)$ oracle calls for $\nabla p(x)$ and $\nabla q(y)$ and $\mathcal{O} \left( \max\left\{\sqrt{\frac{L_p}{\mu_x}}, \sqrt{\frac{L_q}{\mu_y}}, \frac{L_R}{\sqrt{\mu_x \mu_y}} \right\}\log \frac{1}{\varepsilon}\right)$ oracle calls for $\nabla R(x,y)$ to find an $\varepsilon$-solution of the problem. To the best of our knowledge, we are the first to develop optimal algorithm with complexity separation in the case $\mu_x \not = \mu_y$. Also, we apply this algorithm to a bilinear saddle point problem and obtain the optimal complexity for this class of problems. 
\end{abstract}  

\section{Introduction}

In this work, we consider strongly convex and strongly concave saddle point problems (SPPs) with the composite structure:
\begin{equation}\label{problem:main}
    \min_{x \in \mathbb{R}^{d_x}}\max_{y \in \mathbb{R}^{d_y}}p(x) + R(x,y) - q(y),
\end{equation}
where $p(x): \mathbb{R}^{d_x} \to \mathbb{R}$, $q(y): \mathbb{R}^{d_y} \to \mathbb{R}$ are convex and $L_p, L_q$-smooth function respectively and $R(x, y): \mathbb{R}^{d_x} \times \mathbb{R}^{d_y} \to \mathbb{R}$ is $L_R$-smooth, $\mu_x$-strongly convex and $\mu_y$-strongly concave. Both composites $p(x)$, $q(y)$ are not necessarily proximal friendly. Note, that we can also consider $R(x,y)$ to be convex-concave and $p(x)$ is $\mu_x$-strongly convex and $q(y)$ is $\mu_y$-strongly concave. By the transformation $p(x)\to p(x)-\frac{\mu_x}{2}\|x\|^2$, $q(y)\to q(y)+\frac{\mu_y}{2}\|y\|^2$ and $R(x,y)\to R(x,y)+\frac{\mu_x}{2}\|x\|^2-\frac{\mu_y}{2}\|y\|^2$ we can reduce this case to the problem \eqref{problem:main}. 

The lower bounds of iteration complexity for the problem \eqref{problem:main} $ \Omega\left(\left(\sqrt{\frac{L_p}{\mu_x}} + \frac{L_R}{\sqrt{\mu_x \mu_y}} + \sqrt{\frac{L_q}{\mu_y}} \right) \log \frac{1}{\varepsilon}\right)$ was proposed in \cite{zhang2019lower}. In this work we present algorithm that achieve these lower bounds. But the focus of this work is on the composites complexity separation which is a key issue in many applications. Below, we give some prime examples of this. 

\textbf{Distributed optimization.} One of the classic application of the problem \eqref{problem:main} is a decentralized 
distributed optimization over communication network $\mathcal{G} = (\mathcal{V}, \mathcal{E})$ :  \begin{equation}\label{pr:ex_distributed}
    \min_{W\mathbf{x} = 0}\sum_{i=1}^n f_i(x_i),
\end{equation}  
where $\mathbf{x} = [x_1^T, \dots, x_n^T]^T$, $n = |\mathcal{V}|$ is the number of nodes (clients) in $\mathcal{G}$, $\{f_i(x_i)\}_{i=1}^n$ are functions that store on nodes with variables $x_i$.  Also client $i$ can communicate with client $j$ if and only if there is edge in graph $\mathcal{G}$, i.e. $(i, j) \in \mathcal{E}$ and $W$ is a gossip matrix for communication network $\mathcal{G}$ which responsible for communications between nodes. In particular, the Laplacian matrix of $\mathcal{G}$ can be used as $W$. Note, that to solve problem in this formulation we have to use conditional optimization methods. To move to unconditional optimization we use a penalty function $\psi(x)$: 
\begin{align}\label{pr:ex_distributed_1}             
\min_{\mathbf{x} \in \mathbb{R}^{nd}} F(\mathbf{x}) + \psi\left(W\mathbf{x}\right),
\end{align}
where $F(\mathbf{x}) = \sum_{i=1}^n f_i(x_i)$ is $\mu_F$-strongly convex, $L_F$-smooth function and $\psi(\mathbf{y}) = 0$ if $\mathbf{y} = 0$, in otherwise $\psi(\mathbf{y}) = +\infty$. In many practical examples, solving the dual problem to \eqref{pr:ex_distributed_1} preferably. This problem has the following form: 
\begin{align}\label{pr:distributed_dual}              
\min_{\mathbf{x} \in \mathbb{R}^{nd}}\max_{\mathbf{y} \in \mathbb{R}^{nd}} F(\mathbf{x}) 
 + \langle \mathbf{y}, W\mathbf{x} \rangle
\end{align}
due to $\psi(\mathbf{y})$ is indicator function and $\psi^*(\mathbf{y}) = 0$. Meanwhile,  to solve  problem \eqref{pr:distributed_dual} we need to find gradient $\nabla F(x)$, i.e. compute local gradients $\nabla f_i(x_i)$. Also, we need to compute $W\mathbf{x}$, $W\mathbf{y}$ (gradients of saddle part $ \langle \mathbf{y}, W\mathbf{x}  \rangle$), i.e. make the communication round. The usually goal for this problem is reduce the communication rounds \cite{stefano_2020_fed}, \cite{Brown2020LanguageMA}. It means, that separation of the oracle complexities to composite and saddle part in problem \eqref{pr:distributed_dual} is significant problem.  

\textbf{Personalized federated learning.} The other important example of the problem \eqref{problem:main} is personalized federated saddle point problem \cite{smith2017federated,wang2018distributed,li2020decentralized,gorbunov2019optimal}
\begin{equation}\label{problem:pfl}
\min_{x \in \mathbb{R}^{d_x}} \max_{y \in \mathbb{R}^{d_y}} \frac{\lambda}{2}\left\|\sqrt{W}X\right\|^2 + \sum_{m=1}^M f_m(x_m, y_m) -\frac{\lambda}{2}\left\|\sqrt{W}Y\right\|^2,
\end{equation}
where $x_1, \ldots, x_M$ and $y_1,\ldots,y_M$ are interpreted as local models on nodes which are grouped into matrices $X := [x_1, \ldots, x_M]^T$ and $Y := [y_1, \ldots, y_M]^T$. $\lambda > 0$ is the key regularization parameter, which corresponds to the personalising degree of the models and $W$ is the gossip matrix reflecting the properties of the communication graph between the nodes. As mentioned above, the composite gradient oracles $WX, WY$ are responsible for the communications. Since, we are interested in separating and reducing gradient calls of composites.

\textbf{Emperical Risk Minimization.} The other important practical case of the problem \eqref{problem:main} is Emperical Risk Minimization problem. This example comes from machine learning \cite{ShalevShwartz2014UnderstandingML}. This problem has the following form
\begin{equation}\label{problem:emperical}
    \min_{x \in \mathbb{R}^{d_x}} p(x) + q(Bx)
\end{equation}
where $q(x)$ is convex loss function, $B$ is matrix with data features and $p(x)$ is strongly convex regularizer. This problem is equal to the following saddle point problem
\begin{equation}\label{problem:emperical_dual}
    \min_{x \in \mathbb{R}^{d_x}}\max_{y \in \mathbb{R}^{d_y}} p(x) + x^T B y - q^*(y),
\end{equation}
 which can be preferable than problem \eqref{problem:emperical} in many practical applications. For example, in distributed optimization to reduce the communication complexity \cite{JMLR:v20:17-608}. Moreover, the gradients $\nabla p(x)$ or $\nabla q^*(y)$ can be difficult to calculate. In this case, we are interested in separating oracle calls.

These practical examples illustrate the importance of separating oracle complexities which lead to the following research question for the problem \eqref{problem:main}
\begin{equation*}
    \textit{Can we effectively separate oracle complexities for composites and saddle parts?}
\end{equation*}

\begin{minipage}{\textwidth}
\begin{table}[H]
    \centering
    \small
    \label{tab:comparison}
\resizebox{\textwidth}{!}{
\begin{threeparttable}
    \begin{tabular}{|c|c|c|c|c|c|}
    \cline{2-5}
    \multicolumn{1}{c|}{}
      &\textbf{ Reference} & \textbf{Oracle calls of $\nabla p(x), \nabla q(y)$} & \textbf{Oracle calls of $\nabla R(x,y)$ or $B, B^T$}& \textbf{Compl. Sep.} \\
      \cline{2-5}
     \multicolumn{1}{c|}{} &\multicolumn{4}{c|}{\textbf{Strongly convex-strongly concave case}}
    \\\hline
      \multirow{10}{*}{\rotatebox[origin=c]{90}{\textbf{Upper} \quad \quad \quad}} &  Korpelevich, 1974 \cite{Korpelevich1976TheEM} & \multicolumn{2}{c|}{} &
      \\ 
      & Tseng, 2000 \cite{tseng2000} & \multicolumn{2}{c|}{} &
      \\
      & Nesterov and Scimali, 2006 \cite{nesterov2006scrimali}
      & \multicolumn{2}{c|}{} &
      \\
      & Gidel et al., 2018 \cite{gidel2018variational}
     &\multicolumn{2}{c|}{\multirow{-4}{*}{ $\mathcal{O} \left( \left(\frac{L_R +L_p}{\mu_x} + \frac{L_R +L_q}{\mu_y} \right) \log \frac{1}{\varepsilon} \right)$}}  & \multirow{-4}{*} {\xmark}
    \\ \cline{2-5}  
    & Alkousa et al., 2019 \cite{alkousa2019}& $\nabla p(x)$ : $\mathcal{O} \left( \sqrt{\frac{L_p}{\mu_x}}\log \frac{1}{\varepsilon}\right)$, $\nabla q(y)$: $\mathcal{O} \left( \frac{L_R}{\sqrt{\mu_x \mu_y}} \sqrt{\frac{L_q}{\mu_y}}\log^3\frac{1}{\varepsilon}\right)$ &    $\mathcal{O} \left(  \frac{L_R\sqrt{L_R}}{\mu_y\sqrt{\mu_x}} \log^3 \frac{1}{\varepsilon} \right)$ & \cmark 
    \\\cline{2-5}
    & Lin et al., 2020 \cite{lin2020}  
     &\multicolumn{2}{c|}{$\mathcal{O} \left( \frac{L_R + \sqrt{L_pL_q}}{\sqrt{\mu_x \mu_y}} \log^3 \frac{1}{\varepsilon} \right)$}  & \xmark
     \\ \cline{2-5}
     & Wang and Li, 2020 \cite{wang2020}  
     &\multicolumn{2}{c|}{$\mathcal{O} \left( \max\left\{\sqrt{\frac{L_p}{\mu_x}}, \sqrt{\frac{L_q}{\mu_y}}, \sqrt{\frac{L_R L}{\mu_x \mu_y}}\right\} \log^3 \frac{(L_p + L_R)(L_q + L_R)}{\mu_x \mu_y} \log \frac{1}{\varepsilon} \right)$}  & \xmark
    \\ \cline{2-5}
     & Kovalev and Gasnikov, 2022 \cite{FOAM}  
     &\multicolumn{2}{c|}{$\mathcal{O} \left( \frac{L_R + \sqrt{L_pL_q}}{\sqrt{\mu_x \mu_y}} \log \frac{1}{\varepsilon} \right)$}  & \xmark
    \\ \cline{2-5}
     &  Jin et al., 2022  \cite{pmlr-v178-jin22b}
     &\multicolumn{2}{c|}{} & 
     \\&  Li et al., 2022 \cite{li2022stochastic}
     &\multicolumn{2}{c|}{\multirow{-2}{*}{$\mathcal{O} \left( \textcolor{black}{\max}\left\{\sqrt{\frac{L_p}{\mu_x}}, \sqrt{\frac{L_q}{\mu_y}}, \frac{L_R}{\mu_x}, \frac{L_R}{\mu_y}  \right\}\log \frac{1}{\varepsilon} \right)$\footnote{In these papers, results close to the lower bounds were obtained but a slightly different notation was used. For more details see \Cref{discussion:subsec:sc-sc}}}} & \multirow{-2}{*}{\xmark}
     \\ \cline{2-5}
    & \cellcolor{bgcolor} \textbf{This paper} & \cellcolor{bgcolor}$\mathcal{O} \left( \textcolor{black}{\max}\left\{\sqrt{\frac{L_p}{\mu_x}}, \sqrt{\frac{L_q}{\mu_y}}  \right\}\log \frac{1}{\varepsilon} \right)$ &  \cellcolor{bgcolor} { $\mathcal{O} \left( \textcolor{black}{\max}\left\{\sqrt{\frac{L_p}{\mu_x}}, \sqrt{\frac{L_q}{\mu_y}}, \frac{L_R}{\sqrt{\mu_x\mu_y}}  \right\}\log \frac{L_R}{\min\{\mu_x, \mu_y\}}\log \frac{1}{\varepsilon} \right)$} & \cellcolor{bgcolor} \cmark 
    \\ \cline{1-5}
     \multirow{2}{*}{\rotatebox[origin=c]{90}{\textbf{Lower}}} & Zhang et al., 2019 \cite{zhang2019lower} & -- &  $\Omega \left(  \frac{L_R}{\sqrt{\mu_x \mu_y}} \log \frac{1}{\varepsilon}\right)$& -- 
    \\ \cline{2-5}
     & Nesterov, 2004 \cite{nesterov2018lectures} & $\Omega \left( \max\left\{\sqrt{\frac{L_p}{\mu_x}}, \sqrt{\frac{L_q}{\mu_y}}\right\}\log \frac{1}{\varepsilon}\right)$& -- & -- 
    \\\hline
    \multicolumn{1}{c|}{}&\multicolumn{4}{c|}{\textbf{Convex-concave case}}
    \\\hline
    \multirow{8}{*}{\rotatebox[origin=c]{90}{\textbf{Upper}}}
    &   Korpelevich, 1974 \cite{Korpelevich1976TheEM} & \multicolumn{2}{c|}{} &
      \\
      & Tseng, 2000 \cite{tseng2000}
      & \multicolumn{2}{c|}{} &
      \\
      & Monteiro and Svaiter, 2010 \cite{monteiro_svaiter2010} 
     &\multicolumn{2}{c|}{\multirow{-3}{*}{ $\mathcal{O} \left(\frac{L\mathcal{D}^2}{\varepsilon}  \right)$}}  & \multirow{-3}{*} {\xmark}
    \\ \cline{2-5} 
     & Chen et al., 2017 \cite{chen2017} & \multicolumn{2}{c|}{$\mathcal{O} \left( \sqrt{\frac{\max\{L_p ,L_q\}}{\varepsilon}}\mathcal{D} + \frac{L_R}{\varepsilon}\mathcal{D}^2\right)$} & \xmark
    \\ \cline{2-5}
    & Lan and Ouyang, 2021 \cite{lan2021} & $\mathcal{O} \left( \sqrt{\frac{\max\{L_p ,L_q\}}{\varepsilon}}\mathcal{D} \right)$ & $\mathcal{O} \left(  \max\left\{\sqrt{\frac{\max\{L_p,L_q\}}{\varepsilon}}\mathcal{D}, \frac{L_R}{\varepsilon}\mathcal{D}^2\right\}\right)$ & \cmark
    \\ \cline{2-5}
    & \cellcolor{bgcolor} \textbf{This paper} & {\cellcolor{bgcolor}$\mathcal{O}\left(\max\left\{\sqrt{\frac{L_p}{\varepsilon}}\mathcal{D}_x, \sqrt{\frac{L_q}{\varepsilon}}\mathcal{D}_y\right\}\log \frac{1}{\varepsilon}\right)$} &  \cellcolor{bgcolor} { $\mathcal{O}\left(\max\left\{\sqrt{\frac{L_p}{\varepsilon}}\mathcal{D}_x, \sqrt{\frac{L_q}{\varepsilon}}\mathcal{D}_y, \frac{L_R}{\varepsilon}\mathcal{D}_x \mathcal{D}_y\right\}\log^2 \frac{1}{\varepsilon}\right)$} & {\cellcolor{bgcolor} \cmark} 
    \\ \cline{1-5}
     \multirow{2}{*}{\rotatebox[origin=c]{90}{\textbf{Lower}}} & Zhang et al., 2019 \cite{zhang2019lower} & -- &  $\Omega \left(  \frac{L_R}{\varepsilon}\mathcal{D}_x \mathcal{D}_y \right)$& -- 
    \\ \cline{2-5}
     & Nesterov, 2004 \cite{nesterov2018lectures} &  $\Omega\left(\sqrt{\frac{L_p}{\varepsilon}}\mathcal{D}_x + \sqrt{\frac{L_q}{\varepsilon}}\mathcal{D}_y\right)$& -- & -- 
    \\\hline
    \multicolumn{1}{c|}{}&\multicolumn{4}{c|}{\textbf{Bilinear strongly convex-strongly concave case}}
    \\\hline
    \multirow{14}{*}{\rotatebox[origin=c]{90}{\textbf{Upper}}}
    &  Korpelevich, 1976 \cite{Korpelevich1976TheEM}
     &\multicolumn{2}{c|}{}  & 
    \\
    &  Nesterov and Scrimali, 2006 \cite{nesterov2006scrimali}
     &\multicolumn{2}{c|}{}  & 
    \\ 
    & Mokhtari et al., 2020 \cite{mokhtari2020}
     &\multicolumn{2}{c|}{\multirow{-3}{*}{$\mathcal{O} \left( \frac{L}{\min\{\mu_p, \mu_q\}}  \log \frac{1}{\varepsilon} \right)$}}  & \multirow{-3}{*}{\xmark}
    \\ \cline{2-5}
    &  Cohen et al., 2021 \cite{cohen2021}
     &\multicolumn{2}{c|}{$\mathcal{O} \left( \textcolor{black}{\max}\left\{\frac{L_p}{\mu_p}, \frac{L_q}{\mu_q}, \sqrt{\frac{\lambda_{\max}(BB^T)}{\mu_p\mu_q}}  \right\}\log \frac{1}{\varepsilon} \right)$}  & \xmark
    \\ \cline{2-5}
     &  Wang and Li, 2020 \cite{wang2020}
     &\multicolumn{2}{c|}{$\mathcal{O} \left( \textcolor{black}{\max}\left\{\sqrt{\frac{L_p}{\mu_p}}, \sqrt{\frac{L_q}{\mu_q}}, \sqrt{\frac{L\sqrt{\lambda_{\max}(BB^T)}}{\mu_p\mu_q}}  \right\}\log \frac{1}{\varepsilon} \right)$}  & \xmark
    \\ \cline{2-5}
    &  Xie et al., 2021 \cite{xie2021}
     &\multicolumn{2}{c|}{$\mathcal{O} \left( \textcolor{black}{\max}\left\{\sqrt[4]{\frac{L_p^2L_q}{\mu_p^2\mu_q}}, \sqrt[4]{\frac{L_q^2 L_p}{\mu_q^2 \mu_p}}, \sqrt{\frac{\lambda_{\max}(BB^T)}{\mu_p\mu_q}}  \right\}\log \frac{1}{\varepsilon} \right)$}  & \xmark
    \\ \cline{2-5}
    &  Kovalev et al., 2021 \cite{kovalev2022accelerated}
     &\multicolumn{2}{c|}{}  &
    \\
    &  Thekumparampil et al., 2022 \cite{pmlr-v151-thekumparampil22a}
     &\multicolumn{2}{c|}{} & 
    \\  
    &  Jin et al., 2022 \cite{Jin2020:mdp}
     &\multicolumn{2}{c|}{}  & 
     \\  
    &  Du et al., 2022 \cite{du2022}
     &\multicolumn{2}{c|}{}  &
     \\  
    &  Li et al., 2022 \cite{li2022stochastic}
     &\multicolumn{2}{c|}{\multirow{-5}{*}{$\mathcal{O} \left( \textcolor{black}{\max}\left\{\sqrt{\frac{L_p}{\mu_p}}, \sqrt{\frac{L_q}{\mu_q}}, \sqrt{\frac{\lambda_{\max}(BB^T)}{\mu_p\mu_q}}  \right\}\log \frac{1}{\varepsilon} \right)$}}  & \multirow{-5}{*}{\xmark}
    \\\cline{2-5}
    & \cellcolor{bgcolor} & \cellcolor{bgcolor}  &  \cellcolor{bgcolor} { $\mathcal{O} \left( \textcolor{black}{\min}\left\{K_1, K_2  \right\}\log \frac{\sqrt{\lambda_{\max}(BB^T)}}{\min\{\mu_p, \mu_q\}}\log \frac{1}{\varepsilon} \right)$} & \cellcolor{bgcolor}
    \\ & \cellcolor{bgcolor} & \cellcolor{bgcolor} & \cellcolor{bgcolor} $K_1 = \max \left\{\sqrt{\frac{L_p\lambda_{\max}(BB^T)}{\mu_p\lambda_{\min}(BB^T)}}, \sqrt{\frac{L_q\lambda_{\max}(BB^T)}{\mu_q\lambda_{\min}(BB^T)}}\right\}$& \cellcolor{bgcolor}
    \\ & \cellcolor{bgcolor} \multirow{-5}{*}{\textbf{This paper}} & \cellcolor{bgcolor}\multirow{-5}{*}{$\mathcal{O} \left( \max\left\{\sqrt{\frac{L_p}{\mu_p}}, \sqrt{\frac{L_q}{\mu_q}}\right\}\log \frac{1}{\varepsilon}\right)$} &\cellcolor{bgcolor} $K_2 = \max\left\{\sqrt{\frac{L_p}{\mu_p}}, \sqrt{\frac{L_q}{\mu_q}}, \sqrt{\frac{\lambda_{\max}(BB^T)}{\mu_p\mu_q}}  \right\}$ & \cellcolor{bgcolor} \multirow{-5}{*} \cmark
    \\ \cline{1-5}
     \multirow{2}{*}{\rotatebox[origin=c]{90}{\textbf{Lower}}} & Zhang et al., 2019 \cite{zhang2019lower} & -- & $\Omega \left(  \sqrt{\frac{\lambda_{\max}(BB^T)}{\mu_p\mu_q}} \log \frac{1}{\varepsilon}\right)$& -- 
    \\ \cline{2-5}
     & Nesterov, 2004 \cite{nesterov2018lectures}  & $\Omega \left( \max\left\{\sqrt{\frac{L_p}{\mu_p}}, \sqrt{\frac{L_q}{\mu_q}}\right\}\log \frac{1}{\varepsilon}\right)$& -- & -- 
    \\\hline
    \multicolumn{1}{c|}{}&\multicolumn{4}{c|}{\textbf{Affinely constrained minimization case}}
    \\\hline
    \multirow{4}{*}{\rotatebox[origin=c]{90}{\textbf{\ \ \  Upper \ \ \ }}}
    &  Kovalev et al., 2020 \cite{NEURIPS2020_d530d454}
     &\multicolumn{2}{c|}{}  & 

    \\
    &  Kovalev et al., 2021 \cite{kovalev2022accelerated}
    &\multicolumn{2}{c|}{\multirow{-2}{*}{$\mathcal{O} \left( \sqrt{\frac{L_p\lambda_{\max}(BB^T)}{\mu_p\lambda_{\min}(BB^T)}}\log \frac{1}{\varepsilon} \right)$}}  & \multirow{-2}{*}{\xmark}
    \\\cline{2-5}
    & \cellcolor{bgcolor} \textbf{This paper} & \cellcolor{bgcolor} $\mathcal{O} \left( \sqrt{\frac{L_p}{\mu_p}}\log \frac{1}{\varepsilon}\right)$ &  \cellcolor{bgcolor} { $\mathcal{O}\left(\min\left\{\sqrt{\frac{L_p\lambda_{\max}(BB^T)}{\mu_p\lambda_{\min}(BB^T)}}, 
    \max \left\{\sqrt{\frac{L_p}{\mu_p}}, \sqrt{\frac{\lambda_{\max}(BB^T)}{\mu_p\varepsilon}}\mathcal{D}_y\right\} \right\}\log^2\frac{1}{\varepsilon}\right)$}  & \cellcolor{bgcolor} \cmark  
    \\ \cline{1-5}
     \multirow{2}{*}{\rotatebox[origin=c]{90}{\textbf{Lower}}} & 
Salim et al., 2021 \cite{pmlr-v151-salim22a} & -- & $\Omega \left(  \sqrt{\frac{L_p\lambda_{\max}(BB^T)}{\mu_p\lambda_{\min}(BB^T)}}\log \frac{1}{\varepsilon} \right)$& -- 
    \\ \cline{2-5}
     & Nesterov, 2004 \cite{nesterov2018lectures}  & $\Omega \left( \sqrt{\frac{L_p}{\mu_p}}\log \frac{1}{\varepsilon}\right)$& -- & -- 
    \\\hline
        \multicolumn{1}{c|}{}&\multicolumn{4}{c|}{\textbf{Billinear case with linear composites}}
    \\\hline
    \multirow{2}{*}{\rotatebox[origin=c]{90}{\textbf{ Upper }}}
    &  Azizian et al., 2020 \cite{azizian_2020} 
    &\multicolumn{2}{c|}{$\mathcal{O} \left( \sqrt{\frac{\lambda_{\max}(BB^T)}{\lambda_{\min}(BB^T)}}\log \frac{1}{\varepsilon} \right)$}  & \xmark
    \\\cline{2-5}
    & \cellcolor{bgcolor} \textbf{This paper} & \cellcolor{bgcolor} {$\mathcal{O}\left (\log\frac{1}{\varepsilon}\right)$} &  \cellcolor{bgcolor} { $\mathcal{O}\left(\sqrt{\frac{\lambda_{\max}(BB^T)}{\lambda_{\min}(BB^T)}}\log^2\frac{1}{\varepsilon}\right)$}  & \cellcolor{bgcolor}{\cmark} 
    \\ \cline{1-5}
     \multirow{2}{*}{\rotatebox[origin=c]{90}{\textbf{Lower \ \ }}} & & & &
     \\& Ibrahim et al., 2020 \cite{pmlr-v119-ibrahim20a} & -- & $\Omega \left(  \sqrt{\frac{\lambda_{\max}(BB^T)}{\lambda_{\min}(BB^T)}}\log \frac{1}{\varepsilon} \right)$& --
     \\& & & &
     \\ \hline
    \end{tabular}   
    \end{threeparttable}
    }
    \caption{Comparison of our results for finding an $\varepsilon$-solution with other works. In the strongly convex-strongly concave case, convergence is measured by the distance to the solution. In the convex-concave case, convergence is measured in terms of the gap function.  {\em Notation:} $L_p, L_q, L_R$ - smoothness constants of functions $p(x), q(y)$ and $R(x, y)$ respectively, $L = \max \{L_p, L_q, L_R\}$, $\mu_x$ - strong convexity constant of $R(x, y)$ for fixed $y$, $\mu_y$ - strong concavity constant of $R(x, y)$ for fixed $x$,  $\mathcal{D}_x, \mathcal{D}_y$ - constants such that $\|x^*\| \leq \mathcal{D}_x$, $\|y^*\| \leq \mathcal{D}_y$, $\mathcal{D} = \max \{\mathcal{D}_x, \mathcal{D}_y\}$. For {\em bilinear case} $\lambda_{\max}(BB^T) > 0$ and $\lambda_{\min}(BB^T) \geq 0$ are maximum and minimum eigenvalue of $BB^T$ respectively, $L = \max(L_p, L_q, \sqrt{\lambda_{\max}(BB^T)})$, $\mu_p, \mu_q$ - strong convexity and strong concavity constants of $p(x), q(y)$ respectively. }
    \label{table_sum}
\end{table}
\end{minipage}
\subsection{Contributions and related works}
Motivated by this research question we introduce the novel algorithm that achieves the lower bound  on iteration complexity and optimally separate  oracle calls for composite and  saddle part. Below we provide more detailed contributions of this work.

$\bullet$ \textbf{New method.} We develop a novel algorithm (\Cref{ae:alg}) used a new sliding idea. We present the idea and convergence analysis of \Cref{ae:alg}. Moreover, this Algorithm has optimal iteration complexity and optimal complexity separation (see \Cref{table_sum}).

$\bullet$ \textbf{Best rates.} To the best of our knowledge, we are the first who present optimal algorithm for composite saddle point problem with complexity separation for the non-symmetric case (i.e., $\mu_x \not = \mu_y$).

$\bullet$ \textbf{Bilinear case.} We adopt our approach to bilinear saddle point problem. Our Algorithm achieve the lower bounds for iteration and oracle complexities (up to logarithmic factor). Moreover, we formulate these results for distributed optimization.

$\bullet$ \textbf{Experiments.}  Also, we present experiments which is not characteristic of most previous works. We provide numerical experiments for bilinear problem on real-world dataset which show the benefits of our approach.

\subsection{Preliminaries}
In this section, we introduce some notation and necessary assumptions used throughout the paper.

\textbf{Notation.}  We denote by $\|\cdot\|$ the standard Euclidean norms. We say that a function $f$ is $L$-smooth on $\mathbb{R}^{d}$ if its gradient is Lipshitz-continuous, i.e., 
\begin{gather}
    \|\nabla f (x)-\nabla f (y)\|\leq L\|x - y\|, 
\end{gather}
for some $L >0$ and any $x, y \in \mathbb{R}^{d}$. We say that a function $f$ is $\mu$-strongly convex on $\mathbb{R}^{d}$ if, for some $\mu > 0$ and any $x, y \in \mathbb{R}^{d}$ it holds that
\begin{gather}
    f (y) \geq f (x) + \langle\nabla f (x), y - x \rangle + \frac{\mu}{2}\|x - y\|^2.
\end{gather}
We say that a pair $(\hat x, \hat y)$ is an $\varepsilon$-solution to \eqref{problem:main} if $\|\hat x - x^*\|^2 + \|\hat y - y^*\|^2 \leq \varepsilon$ where $(x^*, y^*)$ is solution to \eqref{problem:main}. Also, we define the iteration complexity of an algorithm for solving problem \eqref{problem:main} as the number of iterations the algorithm requires to find an $\varepsilon$-solution of this problem. 
To present the idea of \Cref{ae:alg} we introduce $\text{prox}_{\eta f}(\hat x)$ operator
\begin{equation}\label{pr:prox}
    \text{prox}_{\eta f}(\hat x) = \arg \min_x \left\{f(x) + \frac{1}{2\eta}\|x - \hat x\|^2\right\}.
\end{equation}

Meanwhile, we say that a function $f(x)$ is proximal friendly, if we can compute $\text{prox}_{\eta f}(\hat x)$ (solve problem \eqref{pr:prox}) for any point $\hat x$ explicitly or in $\mathcal{O}(1)$ computer calculations.

Finally, we state the assumptions that we impose on problem \eqref{problem:main}.
\begin{assumption}\label{ass:p}
$p(x): \mathbb{R}^{d_x} \to \mathbb{R}$ is $L_p$-smooth and convex on $\mathbb{R}^{d_x}$.
\end{assumption}
\begin{assumption}\label{ass:q}
$q(y): \mathbb{R}^{d_y} \to \mathbb{R}$ is $L_q$-smooth and convex on $\mathbb{R}^{d_y}$.
\end{assumption}
\begin{assumption}\label{ass:R}
    $R(x,y): \mathbb{R}^{d_x} \times \mathbb{R}^{d_y}
 \to \mathbb{R}$ is $L_R$-smooth on $\mathbb{R}^{d_x} \times \mathbb{R}^{d_y}$, $\mu_x$-strongly convex on $\mathbb{R}^{d_x}$ for fixed $y$ and $\mu_y$-strongly concave on $\mathbb{R}^{d_y}$ for fixed $x$.
\end{assumption}
 
\section{Optimal Algorithm}
In this section, we present the key idea to develop \Cref{ae:alg}. After that we provide iteration and oracle complexities of \Cref{ae:alg}. Finally, we propose these complexities to the problem \eqref{problem:main} with one composite. 
\subsection{Idea}
To understand the idea of \Cref{ae:alg} we temporarily switch from the composite saddle point problem \eqref{problem:main} to the composite minimization problem: 
\begin{equation}\label{pr:min_comp}
    \min_x u(x) + v(x)
\end{equation} with $\mu$-strongly convex function $u(x) + v(x)$ and $L_u, L_v$-smooth functions $u(x), v(x)$ respectively. The basic and natural way to solve this problem is apply Nesterov's Accelerated Gradient Descent \cite{Nesterov1983AMF}, \cite{nesterov2018lectures} in the following form from \cite{nesterov2018lectures}:
\begin{align}
 \begin{matrix}
        x_g^k = \alpha x^k + (1 - \alpha)x_f^k
        \\ x^{k+1} = x^k - \eta ( \nabla u(x_g^k) + \nabla v(x_g^k))
        \\ x_f^{k+1} = x_g^k - \tilde \beta ( \nabla u(x_g^k) + \nabla v(x_g^k)).
    \end{matrix}
\end{align}
This method can be rewrite in equivalent form:
\begin{align}\label{AccGD}
    \begin{matrix}
        x_g^k = \alpha x^k + (1 - \alpha)x_f^k
        \\ x^{k+1} = x^k - \eta ( \nabla u(x_g^k) + \nabla v(x_g^k))
        \\ x_f^{k+1} = x_g^k + \beta ( x^{k+1} - x^k)
    \end{matrix}
\end{align}
with $\beta = \frac{\tilde \beta}{\eta}$. The oracle complexity for this method is $\mathcal{O}\left( \sqrt{\frac{L_u + L_v}{\mu}} \log \frac{1}{\varepsilon} \right) $  oracle calls of $\nabla u(x)$ and $\nabla v(x)$ to find an $\varepsilon$-solution to \eqref{pr:min_comp}. This approach does not allow to separate the oracles' complexities that, as mentioned above, may be important in some cases. For example, if we additionally assume that $v(x)$ is proximal friendly function the lower bounds for problem \eqref{pr:min_comp} are $\Omega \left( \sqrt{\frac{L_u}{\mu}} \log \frac{1}{\varepsilon} \right) $ oracle calls of $\nabla u(x)$. Consequently, for this case Nesterov's Accelerated Gradient Descent is not optimal method since it does not use the effectiveness computation of $\text{prox}_{\eta v}(\cdot)$. Due to we can apply more optimal method Accelerated Proximal Point Algorithm \cite{rokafellar1976}, \cite{guler1991}, \cite{auslender2006}, \cite{tseng2008}, \cite{lewis2016}:
\begin{align*} 
   \begin{matrix}
   x_g^k = \alpha x^k + (1 - \alpha)x_f^k
    \\ x^{k+1} = \text{prox}_{\eta v}(x^k - \eta \nabla u(x_g^k))
    \\ x_f^{k+1} = x_g^k + \beta (x^{k+1} - x^k),
   \end{matrix}
\end{align*}
which based on Nesterov's Accelerated Gradient Descent in form \eqref{AccGD}.
This method required $\mathcal{O}\left(\sqrt{\frac{L_u}{\mu}}\log\frac{1}{\varepsilon}\right)$ oracle calls of $\nabla u(x)$ and $\text{prox}_{\eta v}(\cdot)$ to find an $\varepsilon$-solution to the problem \eqref{pr:min_comp}. Moreover, Accelerated Proximal Point Algorithm does not require smoothness of function $v(x)$ and can be applied to problem \eqref{pr:min_comp} with non-smooth composites.  To adapt this approach for $L_v$-smooth and non-proximal friendly function $v(x)$ we rewrite $x^{k+1} = \text{prox}_{\eta v}(x^k - \eta \nabla u(x_g^k))$ for diffirentiable function $v(x)$ in implicit form:
\begin{equation}\label{eq:x_approx_prox}
    x^{k+1} = - \eta \nabla v(x^{k+1}) + x^k - \eta \nabla u(y^k),
\end{equation}
using the first optimal condition. Meanwhile, for non-proximal friendly function we can compute $\text{prox}_{\eta v}(x^k - \eta \nabla u(y^k))$ approximately with Accelerated Gradient Descent and compute $x^{k+1}$ by \eqref{eq:x_approx_prox} using this solution. Summing up the above, we  get the following  method which based on sliding technique \cite{lan2016gradient}
\begin{align*}
   \begin{matrix}
   x_g^k = \alpha x^k + (1 - \alpha)x_f^k
    \\ \hat x^{k+1} \approx \text{prox}_{\eta v}(x^k - \eta \nabla u(x_g^k)) 
    \\ x^{k+1} = x^k - \eta(\nabla u(x_g^k) + \nabla v(\hat x^{k+1}))
    \\ x_f^{k+1} = x_g^k + \beta(\hat x^{k+1} - x^k).
   \end{matrix}
\end{align*}
Also, we can rewrite $\text{prox}_{\eta v}(\cdot)$ using definition 
\begin{align*}
     \text{prox}_{\eta v}(x^k - \eta \nabla u(x_g^k)) &= \arg \min_{x} \left\{\eta v(x) + \frac{1}{2}\| x - ( x^k - \eta \nabla u (x_g^k)) \|^2\right\}
     \\&= \arg \min_{x} \left\{ v(x) + \langle \nabla u (x_g^k), x  \rangle +  \frac{1}{2 \eta}\| x - x^k \|^2\right\},
\end{align*} and get
\begin{align}\label{method:min_comp}
   \begin{matrix}
   x_g^k = \alpha x^k + (1 - \alpha)x_f^k
    \\ \hat x^{k+1} \approx \arg \min_x\left\{v(x) + \langle \nabla u(x_g^k), x \rangle + \frac{1}{2\eta}\|x - x^k\|^2\right\} 
    \\ x^{k+1} = x^k - \eta(\nabla u(x_g^k) + \nabla v(\hat x^{k+1}))
    \\ x_f^{k+1} = x_g^k + \beta(\hat x^{k+1} - x^k).
   \end{matrix}
\end{align}
The oracle complexity of this method is $\mathcal{O} \left(\sqrt{\frac{L_u}{\mu}}\log\frac{1}{\varepsilon}\right)$ oracle calls of $\nabla u(x)$ and $\mathcal{O} \left(\sqrt{\frac{L_v}{\mu}}\log\frac{1}{\varepsilon}\right)$ oracle calls of $\nabla v(x)$ to find an $\varepsilon$-solution. Note that this method allow us to separate oracle complexities for composite minimization problem. Now we are ready to get back to composite saddle point problems. Note that for problem \eqref{pr:min_comp}  is enough to find point $\hat x$ such that $\nabla u(\hat x) + \nabla v(\hat x) = 0$ while for problem \eqref{problem:main} we find point $(\hat x, \hat y)$ such that $\begin{pmatrix}
    \nabla p(\hat x) + \nabla_x R(\hat x, \hat y)
    \\ \nabla q(\hat y) -\nabla_y R(\hat x,\hat y)
\end{pmatrix} = \begin{pmatrix}
     0
    \\ 0
\end{pmatrix}$. This fact is allow us to adapt approach \eqref{method:min_comp} to problem \eqref{problem:main} with replace  $\nabla u(x)$ on operator $A = \begin{pmatrix}
    \nabla p(x)
    \\ \nabla q(y)
\end{pmatrix}$ and $\nabla v(x)$ on operator $B = \begin{pmatrix}
    \nabla_x R(x,y)
    \\ -\nabla_y R(x,y)
\end{pmatrix}$  and develop \Cref{ae:alg}.
\begin{algorithm}[H]
	\caption{}
	\label{ae:alg}
	\begin{algorithmic}[1]
		\State {\bf Input:} $x^0=x_f^0 \in \mathbb{R}^{d_x}$, $y^0=y_f^0 \in \mathbb{R}^{d_y}$
		\State {\bf Parameters:} $\alpha \in (0,1)$, $\eta_x$, $\eta_y$.
		\For{$k=0,1,2,\ldots, K-1$}
			\State $x_g^k = \alpha x^k + (1-\alpha)  x_f^k$, \ \ \ $y_g^k = \alpha y^k + (1-\alpha)  y_f^k$\label{ae:line:1}
			\State $(\hat x^{k+1}, \hat y^{k+1}) \approx \arg \min_{x \in \mathbb{R}^{d_x}}\arg\max_{y \in \mathbb{R}^{d_y}}A_\eta^k(x, y)$\label{ae:line:2}
                \begin{equation}\label{problem:aux}
                    A_\eta^k(x, y) := \langle\nabla p(x_g^k),x \rangle + \frac{1}{2\eta_x}\|x - x^k\|^2 + R(x, y) - \langle\nabla q(y_g^k),y\rangle - \frac{1}{2\eta_y}\|y - y^k\|^2 
                \end{equation}
			\State $x^{k+1} = x^k - \eta_x \left(\nabla p(x_g^{k}) + \nabla_x R(\hat x^{k+1}, \hat y^{k+1})\right)$\label{ae:line:3}
            \Statex \ \ \ \ \ \  $y^{k+1} = y^k - \eta_y  \left(\nabla q(y_g^{k} ) - \nabla_y R(\hat x^{k+1}, \hat y^{k+1})\right)$
            \State $ x_f^{k+1} = x_g^k +\alpha(\hat x^{k+1} - x^k)$, \ \ \  $ y_f^{k+1} = y_g^k + \alpha(\hat y^{k+1} - y^k)$ \label{ae:line:4}
		\EndFor
		\State {\bf Output:} $x^K$, $y^K$
	\end{algorithmic}
\end{algorithm}

\subsection{Complexity}
In the following theorem we present the linear convergence of \Cref{ae:alg}.
\begin{theorem}\label{th:main}
Consider \Cref{ae:alg} for solving Problem \ref{problem:main} under  Assumptions \ref{ass:p}-\ref{ass:R}, with the following tuning for case $\frac{L_p}{\mu_x} > \frac{L_q}{\mu_y}$: 
\begin{equation}\label{ae:tunung_1}
    \alpha = \min\left\{1,\sqrt{\frac{\mu_x}{L_p}}\right\}, \quad  \eta_x = \min \left\{\frac{1}{3\mu_x},\frac{1}{3 L_p \alpha}\right\}, \quad \eta_y = \frac{\mu_x}{\mu_y}\eta_x,
	\end{equation}
 or for case $\frac{L_q}{\mu_y} > \frac{L_p}{\mu_x}$: 
\begin{equation}\label{ae:tunung_2}
    \alpha = \min\left\{1,\sqrt{\frac{\mu_y}{L_q}}\right\}, \quad  \eta_y = \min \left\{\frac{1}{3\mu_y},\frac{1}{3 L_q \alpha}\right\}, \quad \eta_x = \frac{\mu_y}{\mu_x}\eta_y,
	\end{equation}
and let $(\hat x^{k+1}, \hat y^{k+1})$   in \cref{ae:line:2} satisfy    
	\begin{equation}\label{aux:grad_app}
	\eta_x\| 
            \nabla_x A_{\eta}^k (\hat x^{k+1}, \hat y^{k+1})\|^2  +
            \eta_y\| \nabla_y A_{\eta}^k(\hat x^{k+1}, \hat y^{k+1})\|^2 \leq \frac{1}{6\eta_x}\|\hat x^{k+1} - x^k\|^2 + \frac{1}{6\eta_y}\|\hat y^{k+1} - y^k\|^2.
	\end{equation}
 Then, for any 
\begin{equation}
    K \geq 3 \max \left\{1, \sqrt{\frac{L_p}{\mu_x}}, \sqrt{\frac{L_q}{\mu_y}}\right\}\log \frac{\frac{1}{\eta_x}\|x^0 - x^*\|^2 + \frac{1}{\eta_y}\|y^0 - y^*\|^2 + \frac{2}{\alpha}\mathrm{D}_p(x_f^0, x^*) + \frac{2}{\alpha}\mathrm{D}_q(y_f^0, y^*)}{\varepsilon}
\end{equation}
we have the following estimate for the distance to the solution $(x^*, y^*)$:
\begin{equation}
    \frac{1}{\eta_x}\|x^K - x^*\|^2 + \frac{1}{\eta_y}\|y^K - y^*\|^2 \leq \varepsilon.
\end{equation}
\end{theorem}
Proof of \Cref{th:main} you can find in \Cref{sec:proof_th1}.

\textbf{Auxiliary subproblem complexity.} 
At each iteration of \Cref{ae:alg} we need to find $\hat x^{k+1}, \hat y^{k+1}$  (solution to the problem \eqref{problem:aux}) that satisfies condition \eqref{aux:grad_app}. $A_{\eta}^k(x, y)$ is $\left(L_R + \frac{1}{\eta_x}\right)$-smooth in $x$
 for fixed $y$ and $\left(L_R + \frac{1}{\eta_y}\right)$-smooth in $y$ for fixed $x$. Due to this property we get the following inequality
 \begin{align*}
    \eta_x &\|\nabla_x A_{\eta}^k (\hat x^{k+1}, \hat y^{k+1})\|^2  + \eta_y\| \nabla_y A_{\eta}^k(\hat x^{k+1}, \hat y^{k+1})\|^2 
    \\&\leq\max\left\{\eta_x\left(L_R + \frac{1}{\eta_x}\right)^2, \eta_y \left(L_R + \frac{1}{\eta_y}\right)^2\right\}\left(\|\hat x^{k+1} - \hat x^{k+1}_*\|^2 + \|\hat y^{k+1} - \hat y^{k+1}_*\|^2\right),
\end{align*}
where $(\hat x^{k+1}_*, \hat y^{k+1}_*)$ is the solution to the problem \eqref{problem:aux}. It means that $\frac{\frac{1}{6\eta_x}\|\hat x^{k+1} - x^k\|^2 + \frac{1}{6\eta_y}\|\hat y^{k+1} - y^k\|^2}{\max\left\{\eta_x\left(L_R + \frac{1}{\eta_x}\right)^2, \eta_y \left(L_R + \frac{1}{\eta_y}\right)^2\right\}}$-solution to \eqref{problem:aux} satisfies condition \eqref{aux:grad_app}. To find this solution we can apply the algorithm FOAM (Algorithm 4 from \cite{FOAM}) from starting point $(x^k, y^k)$ and get the following complexity.
\begin{theorem}\label{th:inner_complexity} Algorithm 4 from \cite{FOAM} requires the following number of gradient evaluations: 
\begin{equation}
    T = \mathcal{O}\left(\left(1 + L_R\sqrt{\eta_x\eta_y}\right)\log \frac{1}{\gamma}\right)
\end{equation}
to find an $\gamma$-accurate solution of problem \eqref{problem:aux} with
\begin{equation}\label{gamma}
    \gamma = \frac{\frac{1}{6\eta_x}\|\hat x^{k+1} - x^k\|^2 + \frac{1}{6\eta_y}\|\hat y^{k+1} - y^k\|^2}{\max\left\{\eta_x\left(L_R + \frac{1}{\eta_x}\right)^2, \eta_y \left(L_R + \frac{1}{\eta_y}\right)^2\right\}}.
\end{equation}
\end{theorem}
Proof of \Cref{th:inner_complexity} you can find in \Cref{sec:proof_th2}.
\begin{remark}
Note that the stopping criterion \eqref{aux:grad_app} for solving the auxiliary problem \eqref{problem:aux} is practical due to it does not depend on point $(\hat x^{k+1}_*, \hat y^{k+1}_*)$ (solution to \eqref{problem:aux}).
\end{remark}

\textbf{Overall complexity.}  
To formulate the total iterative complexity of \Cref{ae:alg}  we do some mathematical calculations for case $\frac{L_p}{\mu_x} \geq \frac{L_q}{\mu_y}$.
\begin{align*}
    K \times T &= \mathcal{O}\left(\left(1 + \sqrt{\frac{L_p}{\mu_x}}\right)\log \frac{1}{\varepsilon}\right) \times \mathcal{O}\left(\left(1 + L_R\sqrt{\eta_x \eta_y}\right)\log \frac{1}{\gamma}\right)
    \\& = {\mathcal{O}}\left(\left(1 + \sqrt{\frac{L_p}{\mu_x}} + L_R\sqrt{\frac{L_p}{\mu_y}}\eta_x \right)\log \frac{L_R}{\min\{\mu_x, \mu_y\}}\log \frac{1}{\varepsilon}\right)\\& = {\mathcal{O}}\left(\left(1 + \sqrt{\frac{L_p}{\mu_x}} + \frac{L_R}{\sqrt{\mu_x\mu_y}} \right)\log \frac{L_R}{\min\{\mu_x, \mu_y\}}\log \frac{1}{\varepsilon}\right)
\end{align*}
The case $\frac{L_q}{\mu_y} > \frac{L_p}{\mu_x}$ is symmetric. Solving the auxiliary subproblem \eqref{problem:aux} does not require calling oracles $\nabla p(x), \nabla q(y)$. These oracles are called only in \cref{ae:line:3} of \Cref{ae:alg}. Summing up the oracle complexity of \Cref{ae:alg} we present in the following theorem. 
\begin{theorem}\label{th:oracle_complexity}
Consider Problem~\eqref{problem:main} under  \Crefrange{ass:p}{ass:R}. Then, to find an $\varepsilon$-solution,  \Cref{ae:alg} requires
\begin{equation*}
	\mathcal{O}\left(\max\left\{1, \sqrt{\frac{L_p}{\mu_x}}, \sqrt{\frac{L_q}{\mu_y}}\right\}\log\frac{1}{\varepsilon}\right) ~~ \text{calls of }\,\, \nabla p(x), \nabla q(y)
\end{equation*}
and 
\begin{equation*}
    \mathcal{O}\left(\left(\sqrt{\frac{L_p}{\mu_x}} + \sqrt{\frac{L_q}{\mu_y}} + \frac{L_R}{\sqrt{\mu_x \mu_y}}\right)\log \frac{L_R}{\min\{\mu_x, \mu_y\}}\log\frac{1}{\varepsilon}\right) ~~ \text{calls of }\,\, \nabla R(x, y).
\end{equation*} 
\end{theorem}
\textbf{SPP with one composite.}
The important particular case of problem \eqref{problem:main} is composite saddle point problem with one composite. It means that in this case $L_q = 0$ (or $L_p = 0$). By \Cref{th:oracle_complexity} \Cref{ae:alg} requires $\mathcal{O}\left(\sqrt{\frac{L_p}{\mu_x}}\log\frac{1}{\varepsilon}\right)$ oracle calls of $\nabla p(x)$ and $\mathcal{O}\left(\max\left\{\sqrt{\frac{L_p}{\mu_x}}, \frac{L_R}{\sqrt{\mu_x \mu_y}}\right\}\log\frac{1}{\varepsilon}\right)$ oracle calls of $\nabla R(x,y)$ to find an $\varepsilon$-solution to problem \eqref{problem:main}.
\subsection{Convex-concave and strongly convex-concave composite SPP}
For the convex-concave composite SPP we assume that $\mu_x = \mu_y = 0$ that means $R(x, y)$ is convex-concave. For the strongly convex-concave composite SPP we assume that $\mu_x > \mu_y = 0$ that means $R(x, y)$ is strongly convex-concave. To present the results for these problems we make standard assumption that solution $(x^*, y^*)$ is limited, i.e. $\|x^*\|\leq \mathcal{D}_x$, $\|y^*\|\leq \mathcal{D}_y$. We use this assumption and consider problem \eqref{problem:main} with regularization terms. For strongly convex-concave case we regularize function $q(y)$ and consider the problem
\begin{equation}\label{problem:strongly-convex_concave}
\min_{x \in \mathbb{R}^{d_x}}\max_{y \in \mathbb{R}^{d_y}} \left\{p(x) + R(x,y) - q(y) -  \frac{\varepsilon}{12\mathcal{D}_y^2}\|y\|^2\right\}
\end{equation}
instead of the problem \eqref{problem:main}. For the convex-concave case we also add the regularization terms for the functions $p(x)$ and $q(y)$ and consider the problem \begin{equation}\label{problem:convex_concave}
\min_{x \in \mathbb{R}^{d_x}}\max_{y \in \mathbb{R}^{d_y}} \left\{p(x) + \frac{\varepsilon}{16\mathcal{D}_x^2}\|x\|^2 + R(x,y) - q(y) -  \frac{\varepsilon}{16\mathcal{D}_y^2}\|y\|^2\right\}
\end{equation}
instead of the problem \eqref{problem:main}.
To demonstrate the equivalence of problems \eqref{problem:strongly-convex_concave}, \eqref{problem:convex_concave} with regularisation terms to problem \eqref{problem:main} we present the following lemma. 
\begin{lemma}\label{lema:convex-concave}
Consider problem \eqref{problem:main} under \Crefrange{ass:p}{ass:R}. If $\mu_x > 0$, $\mu_y = 0$ (strongly convex-concave case) and  $(\hat{x}, \hat{y})$ is an $\frac{2\varepsilon}{3}$-solution to the problem \eqref{problem:strongly-convex_concave}
or if $\mu_x = 0$, $\mu_y = 0$ (convex-concave case) and $(\hat{x}, \hat{y})$ is an $\frac{\varepsilon}{2}$-solution to the problem
\eqref{problem:convex_concave}
with $\|x^*\|\leq \mathcal{D}_x$, $\|y^*\|\leq \mathcal{D}_y$. Then, $(\hat{x}, \hat{y})$ is an $\varepsilon$-solution to problem \eqref{problem:main}. 
\end{lemma}
Due to this lemma we need to find an $\frac{2\varepsilon}{3}$-solution to the problem \eqref{problem:strongly-convex_concave} or an $\frac{\varepsilon}{2}$-solution to the problem \eqref{problem:convex_concave}. To find them we apply \Cref{ae:alg} with composites $p(x)$, $q(y)$. By \Cref{th:oracle_complexity} \Cref{ae:alg} requires $\mathcal{O}\left(\max\left\{\sqrt{\frac{L_p}{\mu_x}}, \sqrt{\frac{L_q}{\varepsilon}}\mathcal{D}_y, \frac{L_R}{\sqrt{\mu_x\varepsilon}}\mathcal{D}_y\right\}\log \frac{L_R}{\min\{\mu_x, \mu_y\}}\log \frac{1}{\varepsilon}\right)$ oracle calls of $\nabla R(x,y)$ and $\mathcal{O}\left(\max\left\{\sqrt{\frac{L_p}{\mu_x}}, \sqrt{\frac{L_q}{\varepsilon}}\mathcal{D}_y\right\}\log \frac{1}{\varepsilon}\right)$ oracle calls of $\nabla p(x), \nabla q(y)$ to find an $\varepsilon$-solution to the problem  \eqref{problem:main} in the strongly convex-concave case and $\mathcal{O}\left(\max\left\{\sqrt{\frac{L_p}{\varepsilon}}\mathcal{D}_x, \sqrt{\frac{L_q}{\varepsilon}}\mathcal{D}_y, \frac{L_R}{\varepsilon}\mathcal{D}_x \mathcal{D}_y\right\}\log \frac{L_R}{\min\{\mu_x, \mu_y\}}\log \frac{1}{\varepsilon}\right)$ oracle calls of $\nabla R(x,y)$ and $\mathcal{O}\left(\max\left\{\sqrt{\frac{L_p}{\varepsilon}}\mathcal{D}_x, \sqrt{\frac{L_q}{\varepsilon}}\mathcal{D}_y\right\}\log \frac{1}{\varepsilon}\right)$ oracle calls of $\nabla p(x), \nabla q(y)$ to find an $\varepsilon$-solution to \eqref{problem:main} in the convex-concave case.

\section{Bilinear Saddle Point Problems}
In the special case, when $R(x,y) = x^T B y$,  \eqref{problem:main} has been also widely studied, dating at least to the classic work of \cite{chambolle2011first} (imaging inverse problems). Modern applications can be find in decentralized optimization \cite{rogozin2023decentralized,chezhegov2023decentralized}. Quadratic variant of the problem \eqref{problem:main} also appeared in
reinforcement learning \cite{du2017stochastic}. In this section we presents our results for bilinear saddle point problems.  

\subsection{Strongly convex-strongly concave bilinear SPP} The bilinear strongly convex-strongly concave  problem has the following form
\begin{equation}\label{problem:bilinear}
    \min_{x \in \mathbb{R}^{d_x}}\max_{y \in \mathbb{R}^{d_y}} p(x) + x^T B y - q(y).
\end{equation}
To this problem we assume that the following assumptions hold
\begin{assumption}\label{ass:bilinear:p}
   $p(x): \mathbb{R}^{d_x} \to \mathbb{R}$ is $L_p$-smooth and $\mu_p$-strongly convex function 
\end{assumption}
\begin{assumption}
  $q(y): \mathbb{R}^{d_y} \to \mathbb{R}$ is $L_q$-smooth and $\mu_q$-strongly convex function  
\end{assumption}
\begin{assumption}\label{ass:bilinear:B}
   Matrix $B : \mathbb{R}^{d_x} \times \mathbb{R}^{d_y}$ is positive semi-definite. 
\end{assumption}

To apply \Cref{ae:alg} to the problem \eqref{problem:bilinear} we reformulate it as a problem 
\begin{equation*}
    \min_x \max_y \tilde p(x) + \frac{\mu_p}{2}\|x\|^2 + x^T B y - \frac{\mu_q}{2}\|y\|^2 - \tilde q(y)
\end{equation*}
with composites $\tilde p(x) = p(x) - \frac{\mu_p}{2}\|x\|^2$, $\tilde q(y) = q(y) - \frac{\mu_y}{2}\|y\|^2$. 

\textbf{Auxiliary subproblem complexity.} At each iteration of \Cref{ae:alg} we need to find a $\gamma$-solution to the problem
\begin{equation}\label{pr:inner_bilinear}
    \min_x \max_y \langle\nabla \tilde p(x_g^k),x\rangle + \frac{1}{2\eta_x}\|x - x^k\|^2 + \frac{\mu_p}{2}\|x\|^2 + x^T B y - \frac{\mu_q}{2}\|y\|^2 - \frac{1}{2\eta_y}\|y - y^k\|^2 - \langle\nabla \tilde q(y_g^k),y\rangle
\end{equation} 
with $\gamma$ defined in \eqref{gamma}. The simplest way to solve this problem is reformulate it as a minimization problem in $x$ using the first order optimal condition in $y$:
\begin{align*}
    B^Tx &- \mu_q y - \frac{1}{\eta_y}(y-y^k) - \nabla \tilde q (y_g^k) = 0
    \\ y(x) &= \frac{1}{\frac{1}{\eta_y} + \mu_q}\left(B^Tx - \nabla \tilde q(y_g^k) + \frac{1}{\eta_y}y^k\right). 
\end{align*} After reformulation we get the quadratic problem 
\begin{align*}
    \min_x \langle x, Ax \rangle + \langle b, x\rangle + c
\end{align*}
with 
\begin{align*}
    A &= \frac{1}{2}\left(\left(\frac{1}{\eta_x} + \mu_p\right)\left(\frac{1}{\eta_y} + \mu_q\right) I + B B^T \right),
    \\ b &= \nabla \tilde p(x_g^k) - \frac{1}{\eta_y}x^k + \left(1 - \frac{2\eta_y}{1 + \eta_y\mu_q}\right)B\left(\nabla \tilde q(y_g^k) - \frac{1}{\eta_y}y^k\right),
    \\ c &= \frac{\eta_y}{2(1 + \eta_y \mu_q)}\|\nabla \tilde q(y_g^k)\|^2 - \frac{1}{1 + \mu_q \eta_y} \langle \nabla \tilde q(y_g^k), y^k \rangle + \frac{\mu_q(1 - \eta_y \mu_q)}{2(1 + \eta_y \mu_q)^2}\|y^k\|^2.
\end{align*}  This problem can be solved by Nesterov's Accelerated Gradient Descent that requires
\begin{align*}
    T &= \mathcal{O}\left( \sqrt{\frac{\lambda_{\max}\left(A\right)}{\lambda_{\min}\left(A\right)}} \log \frac{1}{\gamma}\right) = \mathcal{O}\left(\sqrt{\frac{\left(\frac{1}{\eta_x} + \mu_p\right)\left(\frac{1}{\eta_y} + \mu_q\right) + \lambda_{\max}\left(BB^T\right)}{\left(\frac{1}{\eta_x} + \mu_p\right)\left(\frac{1}{\eta_y} + \mu_q\right) + \lambda_{\min}\left(BB^T\right)}} \log \frac{1}{\gamma}\right)
    \\&= \mathcal{O}\left( \min \left\{ \sqrt{\frac{ \lambda_{\max}\left(BB^T\right)}{ \lambda_{\min}\left(BB^T\right)}}, \sqrt{1 + \frac{\lambda_{\max}\left(BB^T\right)}{\left(\frac{1}{\eta_x} + \mu_p\right)\left(\frac{1}{\eta_y} + \mu_q\right)}}\right\}\log \frac{\sqrt{\lambda_{\max}(BB^T)}}{\min\{\mu_p, \mu_q\}}\right)
\end{align*}
iterations or calls of oracles $B / B^T$ to find an $\varepsilon$-solution to \eqref{pr:inner_bilinear}.

\textbf{Overall complexity.} Next, we make some computations to get the overall complexity of \Cref{ae:alg} for bilinear case
\begin{align*}
    K \times T = \mathcal{O}\left(\left(1 + \sqrt{\frac{L_p}{\mu_x}} + \sqrt{\frac{L_q}{\mu_y}} \right)\log \frac{1}{\varepsilon}\right) \times 
 \mathcal{O}\left( \min\left\{T_1, T_2\right\}\log \frac{\sqrt{\lambda_{\max}(BB^T)}}{\min\{\mu_p, \mu_q\}}\right)
\end{align*}
where $T_1 = \sqrt{\frac{\lambda_{\max}(BB^T)}{\lambda_{\min}(BB^T)}}$ and $T_2 = \sqrt{1 + \frac{\lambda_{\max}\left(BB^T\right)}{\left(\frac{1}{\eta_x} + \mu_p\right)\left(\frac{1}{\eta_y} + \mu_q\right)}}$. Next we compute $K \times T_1$ and $K \times T_2$ for case $\frac{L_p}{\mu_p} \geq \frac{L_q}{\mu_q}$.
\begin{align*}
    K \times T_1 = \mathcal{O}\left(    \sqrt{\frac{L_p}{\mu_p}}\log \frac{1}{\varepsilon} \times \sqrt{\frac{\lambda_{\max}(BB^T)}{\lambda_{\min}(BB^T)}} \right)
    = \mathcal{O}\left(   \sqrt{\frac{L_p}{\mu_p}} \sqrt{\frac{\lambda_{\max}(BB^T)}{\lambda_{\min}(BB^T)}}  \log \frac{1}{\varepsilon}\right),
\end{align*}
\begin{align*}
    K \times T_2 &= \mathcal{O}\left( \sqrt{\frac{L_p}{\mu_p}}\log \frac{1}{\varepsilon} \times \sqrt{1 + \frac{\lambda_{\max}\left(BB^T\right)}{\left(\frac{1}{\eta_x} + \mu_p\right)\left(\frac{1}{\eta_y} + \mu_q\right)}}\right)
    \\&= \mathcal{O}\left( \sqrt{\frac{L_p}{\mu_p}}\log \frac{1}{\varepsilon} \times \left(1 + \sqrt{\lambda_{\max}(BB^T)\eta_x \eta_y}\right)\right)
    \\& = \mathcal{O}\left(\left( \sqrt{\frac{L_p}{\mu_p}} + \sqrt{\frac{L_p\lambda_{\max}(BB^T)}{\mu_q}}\eta_x \right)\log \frac{1}{\varepsilon}\right)\\& = \mathcal{O}\left(\left(\sqrt{\frac{L_p}{\mu_p}} + \sqrt{\frac{\lambda_{\max}(BB^T)}{\mu_p\mu_q}} \right)\log \frac{1}{\varepsilon}\right).
\end{align*}
The case $\frac{L_q}{\mu_q} \geq \frac{L_p}{\mu_p}$ is done similarly. These calculations allow us to formulate the following theorem about oracle complexities of \Cref{ae:alg} applied to the problem \eqref{problem:bilinear}.
\begin{theorem}\label{th:oracle_complexity_bilinear}
Consider Problem~\eqref{problem:bilinear} under  \Crefrange{ass:bilinear:p}{ass:bilinear:B}. Then, to find an $\varepsilon$-solution,  \Cref{ae:alg} requires
\begin{equation*}
	\mathcal{O}\left(\max\left\{1, \sqrt{\frac{L_p}{\mu_p}}, \sqrt{\frac{L_q}{\mu_q}}\right\}\log\frac{1}{\varepsilon}\right) ~~ \text{calls of }\,\, \nabla p(x), \nabla q(y)
\end{equation*}
and 
\begin{equation*}
    \mathcal{O}\left(\min\left\{K_1, K_2\right\}\log \frac{\sqrt{\lambda_{\max}(BB^T)}}{\min\{\mu_p, \mu_q\}}\log\frac{1}{\varepsilon}\right) ~~ \text{calls of }\,\, B \ \text{or} \ \ B^T,
\end{equation*}
where 
\begin{equation*}
    K_1 = \max \left\{\sqrt{\frac{L_p\lambda_{\max}(BB^T)}{\mu_p\lambda_{\min}(BB^T)}}, \sqrt{\frac{L_q\lambda_{\max}(BB^T)}{\mu_q\lambda_{\min}(BB^T)}}\right\}
\end{equation*}
and
\begin{equation*}
    K_2 = \max \left\{\sqrt{\frac{L_p}{\mu_p}}, \sqrt{\frac{L_q}{\mu_q}}, \sqrt{\frac{\lambda_{\max}(BB^T)}{\mu_p\mu_q}}\right\}.
\end{equation*}
\end{theorem}

\subsection{Affinely constrained minimization}
This problem has the following form:
\begin{equation}
    \min_{\mathbf{B}x = c} p(x),
\end{equation}
where $c \in \text{range} \mathbf{B}$. Also, $p(x)$ is $\mu_p$-strongly convex function and $\mathbf{B}$ is positive definite $(\text{i. e. }\lambda_{\min}(\mathbf{B}\mathbf{B}^T) > 0)$. This problem is equivalent to saddle point problem:
\begin{equation}\label{pr:bilinear_aff_constr}
    \min_x \max_y p(x) + x^T \mathbf{B} y - y^T c.
\end{equation}
To apply \Cref{ae:alg} to this problem we make regularization and get the following problem 
\begin{equation*}
    \min_x \max_y p(x) + x^T \mathbf{B} y - y^T c - \frac{\varepsilon}{16 \mathcal{D}_y^2}\|y\|^2.
\end{equation*}
By \Cref{lema:convex-concave}, if we find $\frac{2\varepsilon}{3}$-solution to this problem, then we find an $\varepsilon$-solution to \eqref{pr:bilinear_aff_constr}. To find this solution we apply \Cref{ae:alg} and get the following complexity.
\begin{corollary}\label{th:oracle_complexity_bilinear_aff_constr}
Consider Problem~\eqref{pr:bilinear_aff_constr}. Then, to find an $\varepsilon$-solution,  \Cref{ae:alg} requires
\begin{equation*}
	\mathcal{O}\left(\max\left\{1, \sqrt{\frac{L_p}{\mu_p}}\right\}\log\frac{1}{\varepsilon}\right) ~~ \text{calls of }\,\, \nabla p(x)
\end{equation*}
and 
\begin{equation*}
    \mathcal{O}\left(\sqrt{\frac{L_p\lambda_{\max}(BB^T)}{\mu_p\lambda_{\min}(BB^T)}} \log^2\frac{1}{\varepsilon}\right) ~~ \text{calls of }\,\, B \ \text{or} \ \ B^T.
\end{equation*}
\end{corollary}
This corollary is derived from  \Cref{th:oracle_complexity_bilinear} and the fact that $\min \{a, b\} \leq a$.

\subsection{Bilinear problem with linear composites}
In this subsection we consider bilinear problem with linear composites:
\begin{equation}\label{problem:bilinear_lin_composites}
    \min_x \max_y x^T d + x^T \mathbf{B} y - y^T c,
\end{equation}
where matrix $\mathbf{B}$ is positive definite ($\lambda_{\min}(BB^T) = \lambda_{\min}^+(BB^T)$).
As in the previous subsection, we make the regularization to apply \Cref{ae:alg}. The problem \eqref{problem:bilinear_lin_composites} with regularization has the following form:
\begin{equation}\label{problem:bilinear_lin_composites_reg}
    \min_x \max_y \frac{\varepsilon}{16\mathcal{D}_x^2}\|x\|^2 + x^T d + x^T \mathbf{B} y - y^T c - \frac{\varepsilon}{16\mathcal{D}_y^2}\|y\|^2.
\end{equation}
We need to find an $\frac{\varepsilon}{2}$-solution to find an $\varepsilon$-solution to \eqref{problem:bilinear_lin_composites} by \Cref{lema:convex-concave}. To find it we apply \Cref{ae:alg} with the following complexity. 
\begin{corollary}\label{th:oracle_complexity_bilinear_lin_comp}
Consider Problem~\eqref{problem:bilinear_lin_composites}. Then, to find an $\varepsilon$-solution,  \Cref{ae:alg} requires
\begin{equation*}    
    \mathcal{O}\left(\sqrt{\frac{\lambda_{\max}(BB^T)}{\lambda_{\min}(BB^T)}}\log^2\frac{1}{\varepsilon}\right) ~~ \text{calls of }\,\, B \ \text{or} \ \ B^T.
\end{equation*}
\end{corollary}

\section{Discussion Our Results and Related Works}
In this section we discuss the lower bounds and compare iteration and oracle complexities results for \Cref{ae:alg} stated in \Cref{th:oracle_complexity} with related works. 

\subsection{Lower bounds}
 The lower bounds on iteration complexity to the \textit{strongly convex-strongly concave} problem \eqref{problem:main} is $\Omega\left(\max\left\{\sqrt{\frac{L_p}{\mu_x}}, \sqrt{\frac{L_q}{\mu_y}}, \frac{L_R}{\sqrt{\mu_x\mu_y}} \right\}\log\frac{1}{\varepsilon}\right)$. This result was presented in \cite{zhang2019lower}. The special case of the problem \eqref{problem:main} is 
 \begin{equation*}
     \min_x \max_y \frac{1}{2}\|x\|^2 + R(x,y) - \frac{1}{2}\|y\|^2.
 \end{equation*} This means that the lower bounds on the oracle calls of $\nabla R(x,y)$ is $\Omega\left(\frac{L_R}{\sqrt{\mu_x \mu_y}}\log\frac{1}{\varepsilon}\right)$.  Also, the problem \eqref{problem:main} has a special case 
 \begin{equation*}
     \min_x \max_y p(x) + \frac{\mu_x}{2}\|x\|^2 - q(y) - \frac{\mu_y}{2}\|y\|^2
 \end{equation*} that separate into two problems $\min_x p(x) + \frac{\mu_x}{2}\|x\|^2$ and $\max_y - q(y) - \frac{\mu_y}{2}\|y\|^2$. The lower bounds to these problems on oracle calls of $\nabla p(x), \nabla q(y)$ is  $\Omega\left(\max\left\{\sqrt{\frac{L_p}{\mu_x}}, \sqrt{\frac{L_q}{\mu_y}}\right\}\log\frac{1} {\varepsilon}\right)$ which was proposed in \cite{nesterov2018lectures}. 
 To sum up, the oracle complexities to the problem \eqref{problem:main} is $\Omega\left(\max\left\{\sqrt{\frac{L_p}{\mu_x}}, \sqrt{\frac{L_q}{\mu_y}}\right\}\log\frac{1} {\varepsilon}\right)$ oracle calls of $\nabla p(x), \nabla q(y)$ and $\Omega\left(\frac{L_R}{\sqrt{\mu_x \mu_y}}\log\frac{1}{\varepsilon}\right)$ oracle calls of $\nabla R(x,y)$. 
 
 For the \textit{bilinear strongly convex-strongly concave} problem \eqref{problem:bilinear} the lower bound on iteration complexity $\Omega\left(\max\left\{\sqrt{\frac{L_p}{\mu_p}}, \sqrt{\frac{L_q}{\mu_q}}, \frac{L_B}{\sqrt{\mu_p\mu_q}}\right\}\log\frac{1}{\varepsilon}\right)$ was also proposed in \cite{zhang2019lower}.
Problem 
\begin{equation*}
    \min_x\max_y p(x) + \sqrt{\mu_p \mu_q} \langle x, y \rangle - q(y)
\end{equation*} is a special case of \eqref{problem:bilinear} with $L_B = \sqrt{\mu_p \mu_q}$ and $B = \sqrt{\mu_p \mu_q}I$. It means that the lower bound on oracle calls of $\nabla p(x), \nabla q(y)$ to problem \eqref{problem:bilinear} is 
$\Omega\left(\max\left\{1,\sqrt{\frac{L_p}{\mu_p}}, \sqrt{\frac{L_q}{\mu_q}}\right\}\log\frac{1}{\varepsilon}\right)$.  Besides, problem 
\begin{equation*}
    \min_x \max_y \frac{\mu_p}{2}\|x\|^2 + x^TBy - \frac{\mu_q}{2}\|y\|^2
\end{equation*} also a special case of \eqref{problem:bilinear} that requires $\Omega\left(\max\left\{1, \frac{L_B}{\sqrt{\mu_p\mu_q}}\right\}\log\frac{1}{\varepsilon}\right)$ oracle calls of $B$ or $B^T$ to find an $\varepsilon$-solution to this problem. 

Using similar reasoning, the following lower bounds can be obtained for \textit{strongly convex-concave} problem \eqref{problem:main}:  $\Omega\left(\max\left\{\sqrt{\frac{L_p}{\mu_x}}, \sqrt{\frac{L_q}{\varepsilon}}\mathcal{D}_y, \frac{L_R}{\sqrt{\mu_x\varepsilon}}\mathcal{D}_y\right\}\log \frac{1}{\varepsilon}\right)$ oracle calls of $\nabla R(x,y)$ and $\Omega\left(\max\left\{\sqrt{\frac{L_p}{\mu_x}}, \sqrt{\frac{L_q}{\varepsilon}}\mathcal{D}_y\right\}\log \frac{1}{\varepsilon}\right)$ oracle calls of $\nabla p(x), \nabla q(y)$ to find an $\varepsilon$-solution.  In the \textit{convex-concave} case the lower bounds are $\Omega\left(\max\left\{\sqrt{\frac{L_p}{\varepsilon}}\mathcal{D}_x, \sqrt{\frac{L_q}{\varepsilon}}\mathcal{D}_y, \frac{L_R}{\varepsilon}\mathcal{D}_x \mathcal{D}_y\right\}\log \frac{1}{\varepsilon}\right)$ oracle calls of $\nabla R(x,y)$ and $\Omega\left(\max\left\{\sqrt{\frac{L_p}{\varepsilon}}\mathcal{D}_x, \sqrt{\frac{L_q}{\varepsilon}}\mathcal{D}_y\right\}\log \frac{1}{\varepsilon}\right)$ oracle calls of $\nabla p(x), \nabla q(y)$. For \textit{affinely constrained minimization} problem \eqref{pr:bilinear_aff_constr} the lower bound on oracle calls of $\nabla p(x)$ is $\Omega\left(\sqrt{\frac{L_p}{\mu_p}}\log\frac{1}{\varepsilon}\right)$ and the lower bound on calls of $B / B^T$ is $\Omega\left(\sqrt{\frac{L_p\lambda_{\max}(BB^T)}{\mu_p\lambda_{\min}(BB^T)}}\log\frac{1}{\varepsilon}\right)$. Also, for \textit{billinear problem with linear composites} \eqref{problem:bilinear_lin_composites} the lower bound on oracle calls of $\nabla p(x), \nabla q(y)$ is $\Omega\left(\log\frac{1}{\varepsilon}\right)$ and the lower bound on calls of $B / B^T$ is $\Omega\left(\sqrt{\frac{\lambda_{\max}(BB^T)}{\lambda_{\min}(BB^T)}}\log\frac{1}{\varepsilon}\right)$.

\subsection{Strongly convex-strongly concave and strongly convex-concave case}\label{discussion:subsec:sc-sc}
For the strongly convex-strongly concave case \Cref{ae:alg} has the following oracle complexity
\begin{equation*}
    \mathcal{O}\left(\max\left\{1, \sqrt{\frac{L_p}{\mu_x}}, \sqrt{\frac{L_q}{\mu_y}}\right\}\log\frac{1}{\varepsilon}\right) ~~ \text{calls of }\,\, \nabla p(x), \nabla q(y)
\end{equation*}
and 
\begin{equation*}
    \mathcal{O}\left(\frac{L_R}{\sqrt{\mu_x \mu_y}}\log \frac{L_R}{\min\{\mu_x, \mu_y\}}\log\frac{1}{\varepsilon}\right) ~~ \text{calls of }\,\, \nabla R(x, y)
\end{equation*} 
to find an $\varepsilon$-solution to \eqref{problem:main}. Also, the iteration complexity of \Cref{ae:alg} is $\mathcal{O}\left(\max\left\{1, \sqrt{\frac{L_p}{\mu_x}}, \sqrt{\frac{L_q}{\mu_y}}, \frac{L_R}{\sqrt{\mu_x \mu_y}}\right\}\log \frac{L_R}{\min\{\mu_x, \mu_y\}}\log\frac{1}{\varepsilon}\right)$. This achieves the lower bounds up to logarithmic factor and improves the results for iteration complexity
\begin{equation*}
    \mathcal{O} \left( \max\left\{\sqrt{\frac{L_p}{\mu_x}}, \sqrt{\frac{L_q}{\mu_y}}, \sqrt{\frac{L_R \max\{L_p, L_q, L_R\}}{\mu_x \mu_y}}\right\} \log^3 \frac{(L_p + L_R)(L_q + L_R)}{\mu_x \mu_y} \log \frac{1}{\varepsilon} \right)
\end{equation*}
from \cite{wang2020}, 
\begin{equation*}
    \mathcal{O} \left( \frac{L_R + \sqrt{L_pL_q}}{\sqrt{\mu_x \mu_y}} \log^3 \frac{1}{\varepsilon} \right)
\end{equation*}
according to  \cite{lin2020},
\begin{equation*}
    \mathcal{O} \left( \frac{L_R + \sqrt{L_pL_q}}{\sqrt{\mu_x \mu_y}} \log \frac{1}{\varepsilon} \right)
\end{equation*}
from \cite{FOAM}
and
\begin{equation*}
    \mathcal{O} \left( \textcolor{black}{\max}\left\{\sqrt{\frac{L_p}{\mu_x}}, \sqrt{\frac{L_q}{\mu_y}}, \frac{L_R}{\mu_x}, \frac{L_R}{\mu_y}  \right\}\log \frac{1}{\varepsilon} \right)
\end{equation*} according to \cite{pmlr-v178-jin22b}. Note, that in works \cite{pmlr-v178-jin22b} considered problem \eqref{problem:main} under Assumptions \ref{ass:p}, \ref{ass:q} and the following assumption on function $R(x, y)$. \begin{assumption}
    $R(x,y)$ is twice differentiable function and $\|\nabla_{xx}R(x,y)\| \leq L^{xx}_R$, $\|\nabla_{yy}R(x,y)\| \leq L^{yy}_R$ and $\|\nabla_{xy}R(x,y)\| \leq L_R^{xy}$, where $\|\cdot\|$ is spectral norm.
\end{assumption}
Using these notation the authors of \cite{pmlr-v178-jin22b} get the following iteration complexity to problem \eqref{problem:main}:
\begin{equation*}
    \mathcal{O} \left( \textcolor{black}{\max}\left\{\sqrt{\frac{L_p}{\mu_x}}, \sqrt{\frac{L_q}{\mu_y}}, \frac{L_R^{xx}}{\mu_x}, \frac{L_R^{yy}}{\mu_y}, \frac{L_R^{xy}}{\sqrt{\mu_x\mu_y}}  \right\}\log \frac{1}{\varepsilon} \right).
\end{equation*}
This iteration complexity achieve the lower bounds if $L_R^{xx} = L_R^{yy} = 0$. 
Meanwhile, \Cref{ae:alg} effectively (achieves the lower bounds) separates the oracle calls for composite functions  $\nabla p(x), \nabla q(y)$ and for saddle part $\nabla R(x,y)$ up to logarithmic factor. 

In work \cite{alkousa2019}, the authors separate the oracle calls for $\nabla p(x)$, $\nabla q(y)$ and $\nabla R(x, y)$ but these bounds not achieves the lower bounds even for iteration complexity. Due to these facts, to the best of our knowledge, \Cref{ae:alg} is the first algorithm that achieves the lower bounds on iteration and separate effectively the oracle calls to \eqref{problem:main}.   For the strongly convex-concave case we get the same results with regularization by changing $\mu_y$ on $\frac{\varepsilon}{\mathcal{D}^2_y}$.

\subsection{Convex-concave case}
For the convex-concave case \Cref{ae:alg} requires
\begin{equation*}
    \mathcal{O}\left(\max\left\{\sqrt{\frac{L_p}{\varepsilon}}\mathcal{D}_x, \sqrt{\frac{L_q}{\varepsilon}}\mathcal{D}_y\right\}\log \frac{1}{\varepsilon}\right)  \ \ \text{calls of} \ \ \nabla p(x), \nabla q(y)
\end{equation*}
 and 
 \begin{equation*}
    \mathcal{O}\left(\max\left\{\sqrt{\frac{L_p}{\varepsilon}}\mathcal{D}_x, \sqrt{\frac{L_q}{\varepsilon}}\mathcal{D}_y, \frac{L_R}{\varepsilon}\mathcal{D}_x \mathcal{D}_y\right\}\log^2 \frac{1}{\varepsilon}\right) \ \ \text{calls of} \ \ \nabla R(x,y)
 \end{equation*}  to find an $\varepsilon$-solution to \eqref{problem:main}. This result achieves the lower bounds on iteration complexity $\Omega\left(\max\left\{\sqrt{\frac{L_p}{\varepsilon}}\mathcal{D}_x, \sqrt{\frac{L_q}{\varepsilon}}\mathcal{D}_y, \frac{L_R}{\varepsilon}\mathcal{D}_x \mathcal{D}_y\right\}\log \frac{1}{\varepsilon}\right)$ up to logarithmic factors and generalizes results  
 \begin{equation*}
    \mathcal{O} \left( \sqrt{\frac{\max\{L_p ,L_q\}}{\varepsilon}}\mathcal{D} \right) \ \ \text{calls of} \ \ \nabla p(x), \nabla q(y) 
 \end{equation*}
 and 
 \begin{equation*}
     \mathcal{O} \left(  \max\left\{\sqrt{\frac{\max\{L_p,L_q\}}{\varepsilon}}\mathcal{D}, \frac{L_R}{\varepsilon}\mathcal{D}^2\right\}\right) \ \ \text{calls of} \ \ \nabla R(x,y)
 \end{equation*}
from \cite{lan2021}, where $\mathcal{D} = \max \left\{\mathcal{D}_x, \mathcal{D}_y\right\}$.

\subsection{Bilinear strongly convex-strongly concave case}
For the bilinear strongly convex-strongly concave case \eqref{problem:bilinear} \Cref{ae:alg} requires 
\begin{equation*}
   \mathcal{O}\left(\max\left\{\sqrt{\frac{L_p}{\mu_p}}, \sqrt{\frac{L_q}{\mu_q}}\right\}\log \frac{1}{\varepsilon}\right) \ \  \text{oracle calls of} \ \  \nabla p(x), \nabla q(y)  
\end{equation*}
and
\begin{equation*}
    \mathcal{O}\left(\max\left\{\sqrt{\frac{L_p}{\mu_p}}, \sqrt{\frac{L_q}{\mu_q}}, \frac{ L_{B}}{\sqrt{\mu_p\mu_q}}\right\}\log\frac{L_B}{\min\{\mu_p, \mu_q\} }\log \frac{1}{\varepsilon}\right) \ \ \text{oracle calls of} \ \  \nabla R(x, y)
\end{equation*}
to find an  $\varepsilon$-solution to \eqref{problem:bilinear}. 
Also, the iteration complexity of \Cref{ae:alg} is $\mathcal{O}\left(\max\left\{\sqrt{\frac{L_p}{\mu_p}}, \sqrt{\frac{L_q}{\mu_q}}, \frac{ L_{B}}{\sqrt{\mu_p\mu_q}}\right\}\log\frac{L_B}{\min\{\mu_p, \mu_q\} }\log \frac{1}{\varepsilon}\right)$. The same results $\mathcal{O}\left(\max\left\{\sqrt{\frac{L_p}{\mu_p}}, \sqrt{\frac{L_q}{\mu_q}}, \frac{ L_{B}}{\sqrt{\mu_p\mu_q}}\right\}\log \frac{1}{\varepsilon}\right)$ on iteration complexity were proposed in works \cite{kovalev2022accelerated}, \cite{pmlr-v151-thekumparampil22a}, \cite{du2022} but the main benefit of our approach is complexity separation.

\subsection{Affinely constrained minimization case}
For the affinely constrained minimization case \eqref{pr:bilinear_aff_constr} \Cref{ae:alg} requires 
\begin{align*}   \mathcal{O}\left(\sqrt{\frac{L_p}{\mu_p}}\log\frac{1}{\varepsilon}\right) \ \  \text{oracle calls of} \ \  \nabla p(x)
\end{align*}
and 
\begin{equation*}
    \mathcal{O}\left(\sqrt{\frac{L_p\lambda_{\max}(BB^T)}{\mu_p\lambda_{\min}(BB^T)}}\log^2\frac{1}{\varepsilon}\right) ~~ \text{calls of }\,\, B \ \text{or} \ \ B^T.
\end{equation*}
This matches the iteration complexity of algorithms from the works \cite{NEURIPS2020_d530d454}, \cite{kovalev2022accelerated} up to logarithmic factor. Note, in these works, the authors achieve the lower bounds \cite{pmlr-v151-salim22a} exactly. But the key idea of \Cref{ae:alg} in separating oracle complexities. 

Meanwhile, we can apply this results to distributed optimization problem \eqref{pr:distributed_dual}. For this problem \Cref{ae:alg} requires $\mathcal{O}\left(\sqrt{\frac{L_F}{\mu_F}}\log\frac{1}{\varepsilon}\right)$ calls of $\nabla F(\mathbf{x})$, i.e. local oracle calls and $\mathcal{O}\left(\sqrt{\frac{L_F \lambda_{\max}(W)}{\mu_F \lambda_{\min}^+(W)}}\log^2\frac{1}{\varepsilon}\right)$ calls of $W$, i.e. communication rounds. \Cref{ae:alg} achieves the lower bounds for distributed optimization \cite{scaman2017optimal} up to logarithmic factor. The optimal method for this problem was proposed in \cite{beznosikov2020decentralized}.

\subsection{Billinear case with linear composites}
For the billinear case with linear composites \eqref{problem:bilinear_lin_composites} \Cref{ae:alg} requires 
\begin{equation*}
    \mathcal{O}\left(\log\frac{1}{\varepsilon}\right) \  \text{oracle calls of } \ \nabla p(x), \nabla q(y)
\end{equation*}
and 
\begin{equation*}
\mathcal{O}\left(\sqrt{\frac{\lambda_{\max}(BB^T)}{\lambda_{\min}(BB^T)}}\log^2\frac{1}{\varepsilon}\right) \  \text{oracle calls of } \ B, B^T.
\end{equation*}
 This results match the iteration complexity from the work \cite{azizian_2020} up to logarithmic factor. In contrast to our results, in work \cite{azizian_2020} the lower bounds \cite{pmlr-v119-ibrahim20a} are achieved.

\newpage
\bibliography{lit}
\bibliographystyle{plain}

\newpage
\appendix

\section{Missing proofs }

\subsection{Notation}
First, we need the following notation. For scalar multiplication with non-Euclidean matrix we use $\langle x, y \rangle_A := x^TAy$. Also, we use the following matrix:
\begin{align*}
   P :=  \begin{pmatrix}
\frac{1}{\eta_x} I & 0 \\
0 & \frac{1}{\eta_y}I
\end{pmatrix}, \ \ P^{-1} :=  \begin{pmatrix}
\eta_x I & 0 \\
0 & \eta_y I
\end{pmatrix}
\end{align*}

\subsection{Proof of \Cref{th:main} }\label{sec:proof_th1}
\begin{lemma}
	 Under Assumptions \ref{ass:p}-\ref{ass:R}, the following inequality holds for \Cref{ae:alg}.
	\begin{align*}
-2\left \langle \begin{pmatrix}
            \nabla p (x_g^{k}) +\nabla_x R(\hat x^{k+1}, \hat y^{k+1}) \\
            \nabla q (y_g^{k}) - \nabla_y R(\hat x^{k+1}, \hat y^{k+1})
            \end{pmatrix}; \begin{pmatrix}
            \hat x^{k+1} - x^*\\
            \hat y^{k+1} - y^*
        \end{pmatrix}\right \rangle \leq&-\frac{2}{\alpha}\left \langle \begin{pmatrix}
            \nabla p(x_g^{k}) - \nabla p(x^*) \\
            \nabla q (y_g^{k}) - \nabla q(y^*)
            \end{pmatrix}; \begin{pmatrix}
             x_f^{k+1} - x_g^k\\
             y_f^{k+1} - y_g^k
            \end{pmatrix}\right \rangle 
            \\&+ \frac{2(1 - \alpha)}{\alpha}(\mathrm{D}_p(x_f^k, x^*) - \mathrm{D}_p(x_g^k, x^*)) 
            \\&+ \frac{2(1 - \alpha)}{\alpha}(\mathrm{D}_q(y_f^k, y^*) - \mathrm{D}_q(y_g^k, y^*))
            \\&-2 \mathrm{D}_p(x_g^{k}; x^*)
            - 2\
            \mathrm{D}_q (y_g^{k};y^*) \\&- \mu_x\|\hat x^{k+1} - x^*\|^2 - \mu_y \|\hat y^{k+1} - y^*\|^2.
\end{align*}
\end{lemma}
\begin{proof}
Using the first-order necessary condition $ \begin{pmatrix}
             \nabla p(x^*) \\
             \nabla q(y^*)
            \end{pmatrix} + \begin{pmatrix}
             \nabla_x R(x^*, y^*) \\
             - \nabla_y R(x^*, y^*)
            \end{pmatrix} = 0 $, $\mu_x$-strong convexity in $x$ and $\mu_y$-strong concavity in $y$ of $R(x,y)$,  we get
\begin{align*}
    -2&\left \langle \begin{pmatrix}
            \nabla p (x_g^{k}) +\nabla_x R(\hat x^{k+1}, \hat y^{k+1}) \\
            \nabla q (y_g^{k}) - \nabla_y R(\hat x^{k+1}, \hat y^{k+1})
            \end{pmatrix}; \begin{pmatrix}
            \hat x^{k+1} - x^*\\
            \hat y^{k+1} - y^*
        \end{pmatrix}\right \rangle =
        \\&=-2\left \langle \begin{pmatrix}
            \nabla p(x_g^{k}) +\nabla_x R(\hat x^{k+1}, \hat y^{k+1}) - \nabla p(x^*) - \nabla_x R(x^*, y^*) \\
            \nabla q (y_g^{k}) - \nabla_y R(\hat x^{k+1}, \hat y^{k+1}) - \nabla q(y^*) + \nabla_y R(x^*, y^*)
            \end{pmatrix} ; \begin{pmatrix}
            \hat x^{k+1} - x^*\\
            \hat y^{k+1} - y^*
        \end{pmatrix}\right \rangle
        \\&\leq -2\left \langle \begin{pmatrix}
            \nabla p(x_g^{k}) - \nabla p(x^*) \\
            \nabla q(y_g^{k}) - \nabla q(y^*)
            \end{pmatrix}; \begin{pmatrix}
            \hat x^{k+1} - x^*\\
            \hat y^{k+1} - y^*
            \end{pmatrix}\right \rangle - \mu_x\|\hat x^{k+1} - x^*\|^2 - \mu_y \|\hat y^{k+1} - y^*\|^2
        \\&= -2\left \langle \begin{pmatrix}
            \nabla p(x_g^{k}) - \nabla p(x^*) \\
            \nabla q(y_g^{k}) - \nabla q(y^*)
            \end{pmatrix}; \begin{pmatrix}
            \hat x^{k+1} - x^k\\
            \hat y^{k+1} - y^k
            \end{pmatrix}\right \rangle -2\left \langle \begin{pmatrix}
            \nabla p(x_g^{k}) - \nabla p(x^*) \\
            \nabla q(y_g^{k}) - \nabla q(y^*)
            \end{pmatrix}; \begin{pmatrix}
             x^{k} - x_g^k\\
             y^{k} - y_g^k
            \end{pmatrix}\right \rangle 
            \\&-2\left \langle \begin{pmatrix}
            \nabla p(x_g^{k}) - \nabla p(x^*) \\
            \nabla q(y_g^{k}) - \nabla q(y^*)
            \end{pmatrix}; \begin{pmatrix}
             x_g^{k} - x^*\\
             y_g^{k} - y^*
            \end{pmatrix}\right \rangle - \mu_x\|\hat x^{k+1} - x^*\|^2 - \mu_y \|\hat y^{k+1} - y^*\|^2.
\end{align*}
Using convexity of $p(x)$ and   $q(y)$, we get
\begin{align*}
-2\left \langle \begin{pmatrix}
            \nabla p (x_g^{k}) +\nabla_x R(\hat x^{k+1}, \hat y^{k+1}) \\
            \nabla q (y_g^{k}) - \nabla_y R(\hat x^{k+1}, \hat y^{k+1})
            \end{pmatrix}; \begin{pmatrix}
            \hat x^{k+1} - x^*\\
            \hat y^{k+1} - y^*
        \end{pmatrix}\right \rangle \leq&-2\left \langle \begin{pmatrix}
            \nabla p(x_g^{k}) - \nabla p(x^*) \\
            \nabla q(y_g^{k}) - \nabla q(y^*)
            \end{pmatrix}; \begin{pmatrix}
            \hat x^{k+1} - x^k\\
            \hat y^{k+1} - y^k
            \end{pmatrix}\right \rangle \\&- 2\left \langle \begin{pmatrix}
            \nabla p(x_g^{k}) - \nabla p(x^*) \\
            \nabla q(y_g^{k}) - \nabla q(y^*)
            \end{pmatrix}; \begin{pmatrix}
            x^{k} - x_g^k\\
            y^{k} - y_g^k
            \end{pmatrix}\right \rangle 
            \\&-2 \mathrm{D}_p(x_g^{k}; x^*)
            - 2\
            \mathrm{D}_q(y_g^{k};y^*) 
            \\&- \mu_x\|\hat x^{k+1} - x^*\|^2 - \mu_y \|\hat y^{k+1} - y^*\|^2.
\end{align*}

Now, we use \cref{ae:line:1} and \cref{ae:line:4} of \Cref{ae:alg} and get

\begin{align*}
-2\left \langle \begin{pmatrix}
            \nabla p (x_g^{k}) +\nabla_x R(\hat x^{k+1}, \hat y^{k+1}) \\
            \nabla q (y_g^{k}) - \nabla_y R(\hat x^{k+1}, \hat y^{k+1})
            \end{pmatrix}; \begin{pmatrix}
            \hat x^{k+1} - x^*\\
            \hat y^{k+1} - y^*
        \end{pmatrix}\right \rangle \leq&-\frac{2}{\alpha}\left \langle \begin{pmatrix}
            \nabla p(x_g^{k}) - \nabla p(x^*) \\
            \nabla q (y_g^{k}) - \nabla q(y^*)
            \end{pmatrix}; \begin{pmatrix}
             x_f^{k+1} - x_g^k\\
             y_f^{k+1} - y_g^k
            \end{pmatrix}\right \rangle 
            \\&+ \frac{2(1 - \alpha)}{\alpha}\left \langle \begin{pmatrix}
            \nabla p(x_g^{k}) - \nabla p(x^*) \\
            \nabla q (y_g^{k}) - \nabla q(y^*)
            \end{pmatrix}; \begin{pmatrix}
            x_f^{k} - x_g^k\\
            y_f^{k} - y_g^k
            \end{pmatrix}\right \rangle 
            \\&-2 \mathrm{D}_p(x_g^{k}; x^*)
            - 2\
            \mathrm{D}_q (y_g^{k};y^*) \\&- \mu_x\|\hat x^{k+1} - x^*\|^2 - \mu_y \|\hat y^{k+1} - y^*\|^2
            \\=&-\frac{2}{\alpha}\left \langle \begin{pmatrix}
            \nabla p(x_g^{k}) - \nabla p(x^*) \\
            \nabla q (y_g^{k}) - \nabla q(y^*)
            \end{pmatrix}; \begin{pmatrix}
             x_f^{k+1} - x_g^k\\
             y_f^{k+1} - y_g^k
            \end{pmatrix}\right \rangle 
            \\&+ \frac{2(1 - \alpha)}{\alpha}(\mathrm{D}_p(x_f^k, x^*) - \mathrm{D}_p(x_g^k, x^*)) 
            \\&+ \frac{2(1 - \alpha)}{\alpha}(\mathrm{D}_q(y_f^k, y^*) - \mathrm{D}_q(y_g^k, y^*))
            \\&-2 \mathrm{D}_p(x_g^{k}; x^*)
            - 2\mathrm{D}_q (y_g^{k};y^*) \\& - \mu_x\|\hat x^{k+1} - x^*\|^2 - \mu_y \|\hat y^{k+1} - y^*\|^2.
\end{align*}
This completes the proof of Lemma.
\end{proof}
\begin{assumption}\label{ass:cnd_nmb}
    $\frac{L_p}{\mu_x} \geq \frac{L_q}{\mu_y}$
\end{assumption}
\begin{lemma}\label{lemma:scl_prod}
	Consider  \Cref{ae:alg} for Problem \ref{problem:main} under Assumptions \ref{ass:p}-\ref{ass:cnd_nmb}, with the following tuning: 
\begin{equation}\label{ae:choice}
    \alpha = \min\left\{1,\sqrt{\frac{\mu_x}{L_p}}\right\}, \quad  \eta_x = \min \left\{\frac{1}{3\mu_x},\frac{1}{3 L_p \alpha}\right\}, \quad \eta_y = \frac{\mu_x}{\mu_y}\eta_x,
	\end{equation}
and let $\hat x^{k+1}$   in \cref{ae:line:2} satisfy    
	\begin{equation}\label{app:aux:grad_app}
	\eta_x\| 
            \nabla_x A_{\eta}^k (\hat x^{k+1}, \hat y^{k+1})\|^2  +
            \eta_y\| \nabla_y A_{\eta}^k(\hat x^{k+1}, \hat y^{k+1})\|^2 \leq \frac{1}{6\eta_x}\|\hat x^{k+1} - x^k\|^2 + \frac{1}{6\eta_y}\|\hat y^{k+1} - y^k\|^2.
	\end{equation}
Then, the following inequality holds:
	\begin{equation}\label{ae:rec}
		\mathrm{\Psi}^{k+1} \leq \left(1 - \frac{\alpha}{3}\right)\mathrm{\Psi}^k,
	\end{equation}
	where  
	\begin{equation}\label{ae:Psi}
		\mathrm{\Psi}^k := \frac{1}{\eta_x}\|x^k - x^*\|^2 + \frac{1}{\eta_y}\|y^k - y^*\|^2 + \frac{2}{\alpha}\mathrm{D}_p(x_f^k, x^*) + \frac{2}{\alpha}\mathrm{D}_q(y_f^k, y^*).
	\end{equation}
\end{lemma}
\begin{proof}
Using \cref{ae:line:3} of \Cref{ae:alg}, we get
	\begin{align*}
		\left\|\begin{matrix}
            x^{k+1} - x^*\\
            y^{k+1} - y^*
            \end{matrix}\right\|_P^2 
		=&
		\left\|\begin{matrix}
            x^{k} - x^*\\
            y^{k} - y^*
            \end{matrix}\right\|_P^2
		+ 2\left \langle \begin{pmatrix}
            x^{k+1} - x^k\\
            y^{k+1} - y^k
            \end{pmatrix}; \begin{pmatrix}
            x^{k} - x^*\\
            y^{k} - y^*
            \end{pmatrix}\right \rangle_P
		+\left\|\begin{matrix}
            x^{k+1} - x^k\\
            y^{k+1} - y^k
            \end{matrix}\right\|_P^2
		\\=&
		\left\|\begin{matrix}
            x^{k} - x^*\\
            y^{k} - y^*
            \end{matrix}\right\|_P^2
		-2\left \langle \begin{pmatrix}
            \nabla p (x_g^{k}) +\nabla_x R(\hat x^{k+1}, \hat y^{k+1}) \\
            \nabla q (y_g^{k}) - \nabla_y R(\hat x^{k+1}, \hat y^{k+1})
            \end{pmatrix}; \begin{pmatrix}
            x^{k} - x^*\\
            y^{k} - y^*
            \end{pmatrix}\right \rangle
		+\left\|\begin{matrix}
            x^{k+1} - x^k\\
            y^{k+1} - y^k
            \end{matrix}\right\|_P^2
            \\=&
		\left\|\begin{matrix}
            x^{k} - x^*\\
            y^{k} - y^*
            \end{matrix}\right\|_P^2
		+2\left \langle \begin{pmatrix}
            \nabla p (x_g^{k}) +\nabla_x R(\hat x^{k+1}, \hat y^{k+1}) \\
            \nabla q (y_g^{k}) - \nabla_y R(\hat x^{k+1}, \hat y^{k+1})
            \end{pmatrix}; \begin{pmatrix}
              \frac{\hat x^{k+1} - x^{k}}{\eta_x}\\
             \frac{\hat y^{k+1} -y^{k}}{\eta_y} 
            \end{pmatrix}\right \rangle_{P^{-1}}
        \\&-2\left \langle 
            \begin{pmatrix}
            \nabla p(x_g^k) + \nabla_x R(\hat x^{k+1}, \hat y^{k+1}) \\
            \nabla q (y_g^{k}) - \nabla_y R(\hat x^{k+1}, \hat y^{k+1})
            \end{pmatrix}; \begin{pmatrix}
             \hat x^{k+1} - x^*\\
             \hat y^{k+1} - y^*
        \end{pmatrix}\right \rangle
		+\left\|\begin{matrix}
            x^{k+1} - x^k\\
            y^{k+1} - y^k
            \end{matrix}\right\|_P^2.
    \end{align*}
Since $2\langle a, b\rangle = \|a+b\|^2 - \|a\|^2 - \|b\|^2$, we get
    \begin{align*}
        \left\|\begin{matrix}
            x^{k+1} - x^*\\
            y^{k+1} - y^*
            \end{matrix}\right\|_P^2 
		=&
		\left\|\begin{matrix}
            x^{k} - x^*\\
            y^{k} - y^*
            \end{matrix}\right\|_P^2 + \left \| \begin{matrix}
            \nabla p (x_g^{k}) +\nabla_x R(\hat x^{k+1}, \hat y^{k+1}) + \frac{\hat x^{k+1} - x^{k}}{\eta_x} \\
            \nabla q (y_g^{k}) - \nabla_y R(\hat x^{k+1}, \hat y^{k+1}) + \frac{\hat y^{k+1} - y^{k}}{\eta_y}
            \end{matrix}\right\|_{P^{-1}}^2
		\\&-\left \| \begin{matrix}
            \nabla p (x_g^{k}) +\nabla_x R(\hat x^{k+1}, \hat y^{k+1}) \\
            \nabla q (y_g^{k}) - \nabla_y R(\hat x^{k+1}, \hat y^{k+1})
            \end{matrix}\right\|_{P^{-1}}^2 -\left\|\begin{matrix}
              \hat x^{k+1} - x^{k}\\
             \hat y^{k+1} -y^{k} 
            \end{matrix}\right \|^2_{P}
        \\&-2\left \langle \begin{pmatrix}
            \nabla p (x_g^{k}) +\nabla_x R(\hat x^{k+1}, \hat y^{k+1}) \\
            \nabla q (y_g^{k}) - \nabla_y R(\hat x^{k+1}, \hat y^{k+1})
            \end{pmatrix}; \begin{pmatrix}
             \hat x^{k+1} - x^*\\
             \hat y^{k+1} - y^*
        \end{pmatrix}\right \rangle
		+\left\|\begin{matrix}
            x^{k+1} - x^k\\
            y^{k+1} - y^k
            \end{matrix}\right\|_P^2.
    \end{align*}
Using \cref{ae:line:3} of \Cref{ae:alg}, we get
    \begin{align*}
        \left\|\begin{matrix}
            x^{k+1} - x^*\\
            y^{k+1} - y^*
            \end{matrix}\right\|_P^2 
		=&
		\left\|\begin{matrix}
            x^{k} - x^*\\
            y^{k} - y^*
            \end{matrix}\right\|_P^2 + \left \| \begin{matrix}
            \nabla p (x_g^{k}) +\nabla_x R(\hat x^{k+1}, \hat y^{k+1}) + \frac{\hat x^{k+1} - x^{k}}{\eta_x} \\
            \nabla q (y_g^{k}) - \nabla_y R(\hat x^{k+1}, \hat y^{k+1}) + \frac{\hat y^{k+1} - y^{k}}{\eta_y}
            \end{matrix}\right\|_{P^{-1}}^2
		-\left \| \begin{matrix}
            x^{k+1} - x^k \\
            y^{k+1} - y^k
            \end{matrix}\right\|_{P}^2 -\left\|\begin{matrix}
              \hat x^{k+1} - x^{k}\\
             \hat y^{k+1} -y^{k} 
            \end{matrix}\right \|^2_{P}
        \\&-2\left \langle \begin{pmatrix}
            \nabla p (x_g^{k}) +\nabla_x R(\hat x^{k+1}, \hat y^{k+1}) \\
            \nabla q (y_g^{k}) - \nabla_y R(\hat x^{k+1}, \hat y^{k+1})
            \end{pmatrix}; \begin{pmatrix}
            \hat x^{k+1} - x^*\\
            \hat y^{k+1} - y^*
        \end{pmatrix}\right \rangle
		+\left\|\begin{matrix}
            x^{k+1} - x^k\\
            y^{k+1} - y^k
            \end{matrix}\right\|_P^2
            \\ = &
		\left\|\begin{matrix}
            x^{k} - x^*\\
            y^{k} - y^*
            \end{matrix}\right\|_P^2 + 2 \left \| \begin{matrix}
            \nabla_x A_{\eta}^k (\hat x^{k+1}, \hat y^{k+1})  \\
            - \nabla_y A_{\eta}^k(\hat x^{k+1}, \hat y^{k+1}) 
            \end{matrix}\right\|_{P^{-1}}^2
		 -\left\|\begin{matrix}
              \hat x^{k+1} - x^{k}\\
             \hat y^{k+1} -y^{k} 
            \end{matrix}\right \|^2_{P}
        \\&-2\left \langle \begin{pmatrix}
            \nabla p (x_g^{k}) +\nabla_x R(\hat x^{k+1}, \hat y^{k+1}) \\
            \nabla q (y_g^{k}) - \nabla_y R(\hat x^{k+1}, \hat y^{k+1})
            \end{pmatrix}; \begin{pmatrix}
             \hat x^{k+1} - x^*\\
             \hat y^{k+1} - y^*
        \end{pmatrix}\right \rangle.
    \end{align*} 

Using \Cref{lemma:scl_prod} and \cref{ae:line:4} of \Cref{ae:alg} , we get
\begin{align*}
    \left\|\begin{matrix}
            x^{k+1} - x^*\\
            y^{k+1} - y^*
            \end{matrix}\right\|_P^2 
		\leq&\left\|\begin{matrix}
            x^{k} - x^*\\
            y^{k} - y^*
            \end{matrix}\right\|_P^2 + 2 \left \| \begin{matrix}
            \nabla_x A_{\eta}^k (\hat x^{k+1}, \hat y^{k+1})  \\
            - \nabla_y A_{\eta}^k(\hat x^{k+1}, \hat y^{k+1}) 
            \end{matrix}\right\|_{P^{-1}}^2-\frac{2}{3}\left\|\begin{matrix}
              \hat x^{k+1} - x^{k}\\
             \hat y^{k+1} -y^{k} 
            \end{matrix}\right \|^2_{P}
            \\&-\frac{2}{\alpha}\left( \left\langle 
            \nabla p(x_g^{k}) - \nabla p(x^*) ; 
             x_f^{k+1} - x_g^k \right\rangle + \frac{1}{6 \alpha\eta_x}\|x_f^{k+1} 
            - x_g^k \|^2\right)
            \\&- \frac{2}{\alpha}\left( \left\langle
            \nabla q (y_g^{k}) - \nabla q(y^*); 
             y_f^{k+1} - y_g^k
             \right\rangle + \frac{1}{6\alpha\eta_y} \|y_f^{k+1} - y_g^k \|^2\right)  
            \\&+ \frac{2(1 - \alpha)}{\alpha}(\mathrm{D}_p(x_f^k, x^*) - \mathrm{D}_p(x_g^k, x^*)) + \frac{2(1 - \alpha)}{\alpha}(\mathrm{D}_q(y_f^k, y^*) - \mathrm{D}_q(y_g^k, y^*))
            \\&-2 \mathrm{D}_p(x_g^{k}; x^*)
            - 2\
            \mathrm{D}_q (y_g^{k};y^*) - \mu_x\|\hat x^{k+1} - x^*\|^2 - \mu_y \|\hat y^{k+1} - y^*\|^2.
\end{align*}
Since $\eta_x \leq \frac{1}{3L_p\alpha} $, $\eta_y \leq \frac{\mu_x}{\mu_y} \frac{1}{3L_p\alpha} \leq \frac{1}{3L_q\alpha} $ (by \eqref{ae:choice} and \Cref{ass:cnd_nmb})
\begin{align*}
    \left\|\begin{matrix}
            x^{k+1} - x^*\\
            y^{k+1} - y^*
            \end{matrix}\right\|_P^2 
		\leq&\left\|\begin{matrix}
            x^{k} - x^*\\
            y^{k} - y^*
            \end{matrix}\right\|_P^2 + 2 \left \| \begin{matrix}
            \nabla_x A_{\eta}^k (\hat x^{k+1}, \hat y^{k+1})  \\
            - \nabla_y A_{\eta}^k(\hat x^{k+1}, \hat y^{k+1}) 
            \end{matrix}\right\|_{P^{-1}}^2
		 -\frac{2}{3}\left\|\begin{matrix}
              \hat x^{k+1} - x^{k}\\
             \hat y^{k+1} -y^{k} 
            \end{matrix}\right \|^2_{P}
            \\&-\frac{2}{\alpha}\left( \left\langle 
            \nabla p(x_g^{k}) - \nabla p(x^*) ; 
             x_f^{k+1} - x_g^k \right\rangle + \frac{L_p}{2}\|x_f^{k+1} 
            - x_g^k \|^2\right)
            \\&- \frac{2}{\alpha}\left( \left\langle
            \nabla q (y_g^{k}) - \nabla q(y^*); 
             y_f^{k+1} - y_g^k
             \right\rangle + \frac{L_q}{2} \|y_f^{k+1} - y_g^k \|^2\right) 
            \\&+ \frac{2(1 - \alpha)}{\alpha}(\mathrm{D}_p(x_f^k, x^*) - \mathrm{D}_p(x_g^k, x^*)) + \frac{2(1 - \alpha)}{\alpha}(\mathrm{D}_q(y_f^k, y^*) - \mathrm{D}_q(y_g^k, y^*))
            \\&-2 \mathrm{D}_p(x_g^{k}; x^*)
            - 2\
            \mathrm{D}_q (y_g^{k};y^*) - \mu_x\|\hat x^{k+1} - x^*\|^2 - \mu_y \|\hat y^{k+1} - y^*\|^2.
\end{align*}
 $L_p$-smoothness of $p(x)$ and $L_q$-smoothness of $q(y)$ gives
\begin{align*}
    \left\|\begin{matrix}
            x^{k+1} - x^*\\
            y^{k+1} - y^*
            \end{matrix}\right\|_P^2 
		\leq&\left\|\begin{matrix}
            x^{k} - x^*\\
            y^{k} - y^*
            \end{matrix}\right\|_P^2 + 2 \left \| \begin{matrix}
            \nabla_x A_{\eta}^k (\hat x^{k+1}, \hat y^{k+1})  \\
            - \nabla_y A_{\eta}^k(\hat x^{k+1}, \hat y^{k+1}) 
            \end{matrix}\right\|_{P^{-1}}^2
		 -\frac{2}{3}\left\|\begin{matrix}
              \hat x^{k+1} - x^{k}\\
             \hat y^{k+1} -y^{k} 
            \end{matrix}\right \|^2_{P}
            \\&-\frac{2}{\alpha}(\mathrm{D}_p(x_f^{k+1}, x^*) - \mathrm{D}_p(x_g^k, x^*)) - \frac{2}{\alpha}(\mathrm{D}_q(y_f^{k+1}, y^*) - \mathrm{D}_q(y_g^k, y^*)) 
            \\&+ \frac{2(1 - \alpha)}{\alpha}(\mathrm{D}_p(x_f^k, x^*) - \mathrm{D}_p(x_g^k, x^*)) + \frac{2(1 - \alpha)}{\alpha}(\mathrm{D}_q(y_f^k, y^*) - \mathrm{D}_q(y_g^k, y^*))
            \\&-2 \mathrm{D}_p(x_g^{k}; x^*)
            - 2\
            \mathrm{D}_q (y_g^{k};y^*) - \mu_x\|\hat x^{k+1} - x^*\|^2 - \mu_y \|\hat y^{k+1} - y^*\|^2
        \\=&\left\|\begin{matrix}
            x^{k} - x^*\\
            y^{k} - y^*
            \end{matrix}\right\|_P^2 + 2\left(  \eta_x\| 
            \nabla_x A_{\eta}^k (\hat x^{k+1}, \hat y^{k+1})\|^2  +
            \eta_y\| \nabla_y A_{\eta}^k(\hat x^{k+1}, \hat y^{k+1})\|^2\right.
            \\&\left.- \frac{1}{6\eta_x}\|\hat x^{k+1} - x^k\|^2 - \frac{1}{6\eta_y}\|\hat y^{k+1} - y^k\|^2\right)
            -\frac{2}{\alpha}\mathrm{D}_p(x_f^{k+1}, x^*) - \frac{2}{\alpha}\mathrm{D}_q(y_f^{k+1}, y^*) 
            \\&+ \frac{2(1 - \alpha)}{\alpha}\mathrm{D}_p(x_f^k, x^*) + \frac{2(1 - \alpha)}{\alpha}\mathrm{D}_q(y_f^k, y^*)
            \\& - \mu_x\left(\|\hat x^{k+1} - x^*\|^2 + \frac{1}{3\mu_x\eta_x}\|\hat x^{k+1} - x^k\|^2\right) - \mu_y \left(\|\hat y^{k+1} - y^*\|^2 + \frac{1}{3\eta_y\mu_y}\|\hat y^{k+1} - y^k\|^2\right).
\end{align*}

Since $\eta_x \leq \frac{1}{3\mu_x}$, $\eta_y = \frac{\mu_x}{\mu_y} \eta_x \leq \frac{1}{3\mu_y}$ (by \eqref{ae:choice}). Using inequality $-\|a - b\|^2 \geq - 2\|a\|^2 - 2\|b\|^2$, we get
\begin{align*}
    \left\|\begin{matrix}
            x^{k+1} - x^*\\
            y^{k+1} - y^*
            \end{matrix}\right\|_P^2 
		\leq&\left\|\begin{matrix}
            x^{k} - x^*\\
            y^{k} - y^*
            \end{matrix}\right\|_P^2 + 2\left(  \eta_x\| 
            \nabla_x A_{\eta}^k (\hat x^{k+1}, \hat y^{k+1})\|^2  +
            \eta_y\| \nabla_y A_{\eta}^k(\hat x^{k+1}, \hat y^{k+1})\|^2 - \frac{1}{6\eta_x}\|\hat x^{k+1} - x^k\|^2 \right.
            \\&\left.- \frac{1}{6\eta_y}\|\hat y^{k+1} - y^k\|^2\right)
            -\frac{2}{\alpha}\mathrm{D}_p(x_f^{k+1}, x^*) - \frac{2}{\alpha}\mathrm{D}_q(y_f^{k+1}, y^*)  + \frac{2(1 - \alpha)}{\alpha}\mathrm{D}_p(x_f^k, x^*) 
            \\&+ \frac{2(1 - \alpha)}{\alpha}\mathrm{D}_q(y_f^k, y^*) - \mu_x\| x^{k} - x^*\|^2  - \mu_y \| y^{k} - y^*\|^2.
\end{align*}

Since \eqref{app:aux:grad_app}, we get 
\begin{align*}
    \frac{1}{\eta_x} \|x^{k+1} - x^*\|^2 +& \frac{1}{\eta_y}
            \|y^{k+1} - y^*\|^2
		+\frac{2}{\alpha}\mathrm{D}_p(x_f^{k+1}, x^*) + \frac{2}{\alpha}\mathrm{D}_q(y_f^{k+1}, y^*) \leq
  \\\leq&\frac{1}{\eta_x}\left(1 - \mu_x\eta_x\right)\|x^{k} - x^*\|^2 + \frac{1}{\eta_y}\left(1 - \mu_y \eta_y\right)\|y^{k} - y^*\|^2
 \\&+ \frac{2(1 - \alpha)}{\alpha}\mathrm{D}_p(x_f^k, x^*) + \frac{2(1 - \alpha)}{\alpha}\mathrm{D}_q(y_f^k, y^*)
 \\\leq&(1 - \alpha)\left[\frac{1}{\eta_x}\|x^{k} - x^*\|^2 + \frac{1}{\eta_y}\|y^{k} - y^*\|^2
 + \frac{2}{\alpha}\mathrm{D}_p(x_f^k, x^*) + \frac{2}{\alpha}\mathrm{D}_q(y_f^k, y^*)\right].
\end{align*}
In the last inequality we use that $\alpha > \frac{\alpha}{3}$, $\eta_x \mu_x \geq \frac{\alpha}{3}$ and $\eta_y\mu_y \geq \frac{\alpha}{3}$. If $L_p \leq \mu_x$, then $\alpha = 1$, $\eta_x \mu_x = \frac{1}{3} = \frac{\alpha}{3}$, $\eta_y \mu_y = \frac{1}{3} = \frac{\alpha}{3}$. If $L_p > \mu_x$, then $\alpha = \sqrt{\frac{\mu_x}{L_p}}$, $\eta_x \mu_x = \sqrt{\frac{\mu_x}{3L_p}} \geq \frac{\alpha}{3}$, $\eta_y \mu_y = \sqrt{\frac{\mu_x}{3L_p}} \geq \frac{\alpha}{3}$.

By \eqref{ae:Psi} definition of $\mathrm{\Psi}^k$, we get
\begin{align*}
    \mathrm{\Psi}^{k+1} \leq \left(1 - \frac{\alpha}{3}\right) \mathrm{\Psi}^k.
\end{align*}
\end{proof}

\textbf{Proof of \Cref{th:main}}
Using the property of the Bregman devergence $\mathrm{D}_f(x, y) \geq 0 $ and running the recursion \eqref{ae:rec} we get
\begin{equation*}
    \frac{1}{\eta_x}\|x^K - x^*\|^2 + \frac{1}{\eta_y}\|y^K - y^*\|^2 \leq \mathrm{\Psi}^K \leq C \left(1 - \frac{\alpha}{3}\right)^K,
\end{equation*}
where $C$ is defined as
\begin{equation*}
    C = \frac{1}{\eta_x}\|x^0 - x^*\|^2 + \frac{1}{\eta_y}\|y^0 - y^*\|^2 + \frac{2}{\alpha}\mathrm{D}_p(x_f^0, x^*) + \frac{2}{\alpha}\mathrm{D}_q(y_f^0, y^*).
\end{equation*}
After $K = \frac{3}{\alpha} \log \frac{1}{\varepsilon}$ iterations of \Cref{ae:alg} we get a pair $(x^K, y^K)$ satisfies the following inequality
\begin{equation*}
    \frac{1}{\eta_x}\|x^K - x^*\|^2 + \frac{1}{\eta_y}\|y^K - y^*\|^2 \leq \varepsilon.
\end{equation*}

\subsection{Proof of \Cref{th:inner_complexity}}\label{sec:proof_th2}

\begin{lemma}\label{lemma:var_uv}
    Consider function $R(x, y)$ under \Cref{ass:R}. If we make the following replacing variables $x = \alpha u$, $y = \beta v$, then function $\tilde R(u, v) := R(x, y)$ is $\tilde{L}$-smooth, $\mu_u$-strongly convex in $u$ with fixed $v$ and $\mu_v$-strongly concave in $v$ for fixed $u$, with $\tilde L = \max\{\alpha^2, \beta^2\} L$, $\mu_u = \alpha^2 \mu_x$, $\mu_v = \beta^2\mu_y $
\end{lemma}
\begin{proof}
Firstly, let us consider that 
\begin{align*}
    \nabla_x R(x, y) = \begin{pmatrix} \frac{\partial R(x, y)}{\partial x_1} 
    \\ \dots
    \\ \frac{\partial R(x, y)}{\partial x_{d_x}}
    \end{pmatrix} = \begin{pmatrix} \frac{\partial R(\alpha u, y)}{\partial u_1} \cdot \frac{\partial u_1}{\partial x_1}
    \\ \dots
    \\ \frac{\partial R(\alpha u, y)}{\partial u_{d_x}} \cdot \frac{\partial u_{d_x}}{\partial x_{d_x}}
    \end{pmatrix} = \begin{pmatrix} \frac{\partial R(\alpha u, y)}{\partial u_1} \cdot \frac{1}{\alpha}
    \\ \dots
    \\ \frac{\partial R(\alpha u, y)}{\partial u_{d_x}} \cdot \frac{1}{\alpha}
    \end{pmatrix}  = \frac{1}{\alpha} \nabla_u \tilde R(u, v).
\end{align*}
Using the analogical calculations we get $\nabla_y R(x, y) = \frac{1}{\beta} \nabla_v\tilde R (u, v)$.
Now we define the smoothness constant of function $\tilde R(u, v)$ using $L$-smoothness of function $R(x, y)$. 
\begin{align*}
    \|\nabla \tilde R(u_1, v_1) - \nabla \tilde R(u_2, v_2)\|^2 &= \|\nabla_u \tilde R(u_1, v_1) - \nabla_u \tilde R(u_2, v_2)\|^2 + \|\nabla_v \tilde R(u_1, v_1) - \nabla_v \tilde R(u_2, v_2)\|^2
    \\& = \alpha^2\|\nabla_x R(x_1, y_1) - \nabla_x R(x_2, y_2)\|^2 + \beta^2\|\nabla_y R(x_1, y_1) - \nabla_y R(x_2, y_2)\|^2 
    \\& \leq \max\{\alpha^2, \beta^2\} \|\nabla R(x_1, y_1) - \nabla R(x_2, y_2)\|^2 \\&\leq  \max\{\alpha^2, \beta^2\} L^2 \left(\|x_1 - x_2\|^2 + \|y_1 - y_2\|^2\right)
    \\&=  \max\{\alpha^2, \beta^2\} L^2 \left(\alpha^2\|u_1 - u_2\|^2 + \beta^2\|v_1 - v_2\|^2\right) 
    \\&\leq \tilde{L}^2 \left(\|u_1 - u_2\|^2 + \|v_1 - v_2\|^2\right),
\end{align*} 
with $\tilde L =  \max\{\alpha^2, \beta^2\} L$. 

Now we define $\mu
_u$-strongly convex constant of function $\tilde R(u, v)$ in $u$ for fixed $v$.
\begin{align*}
    \tilde R(u_2, v) = R( x_2, y) &\geq R( x_1, y) + \langle \nabla_x R(x_1 ,y),  x_2 - x_1\rangle + \frac{\mu_x}{2}\|x_2 - x_1\|^2
    \\& = \tilde R( u_1, v) + \left\langle \frac{1}{\alpha} \nabla_u \tilde R(u_1 ,v), \alpha (u_2 - u_1) \right\rangle + \frac{\mu_x\alpha^2}{2}\|u_2 - u_1\|^2
    \\& = \tilde R( u_1, v) + \left \langle \nabla_u \tilde R(u_1 ,v), u_2 - u_1  \right \rangle + \frac{\mu_u}{2}\|u_2 - u_1\|^2,
\end{align*}
with $\mu_u = \alpha^2 \mu_x$. In this equation we use $\mu_x$-strong convexity of $R(x, y)$ in $x$ for fixed $y$ and differentiation rule of complex function. Similarly we get $\mu_v$-strong concavity of $\tilde R(u, v)$ in $v$ for fixed $u$, with $\mu_v = \beta^2 \mu_y$. 
\end{proof}

\textbf{Proof of \Cref{th:inner_complexity}}
Firstly, we make the following replacing variables $x = \alpha u$, $y = \beta v$ in the problem \eqref{problem:aux}. After that we get the following problem in new variables:
\begin{equation}\label{pr:aux_uv}
    \min_u \max_v \alpha\langle\nabla p(x_g^k),u \rangle + \frac{\alpha^2}{2\eta_x}\left\|u - \frac{x^k}{\alpha}\right\|^2 + \tilde R(u, v) - \beta\langle\nabla q(y_g^k), v\rangle - \frac{\beta^2}{2\eta_y}\left\|v - \frac{y^k}{\beta}\right\|^2. 
\end{equation}
By Corollary 1 from \cite{FOAM}  Algorithm FOAM (Algorithm 4 from \cite{FOAM}) requires the following number of gradient evaluations: 
\begin{equation}
    T = \mathcal{O}\left(\left(\frac{\tilde L_R + \frac{\alpha^2}{\eta_x} + \frac{\beta^2}{\eta_y}}{\sqrt{(\mu_u + \frac{\alpha^2}{\eta_x})(\mu_v + \frac{\beta^2}{\eta_y})}}\right)\log \frac{1}{\gamma}\right)
\end{equation}
to find an $\gamma$-accurate solution of problem \eqref{pr:aux_uv}. By \Cref{lemma:var_uv} we get $\tilde L_R = \max\{\alpha, \beta\} L_R$, $\mu_u = \alpha^2 \mu_x$ and $\mu_v = \beta^2\mu_y$. Using these values we get the following number of gradient evaluations: 
\begin{equation}
    T = \mathcal{O}\left(\left(\frac{\max\{\alpha^2, \beta^2\} L_R + \frac{\alpha^2}{\eta_x} + \frac{\beta^2}{\eta_y}}{\sqrt{(\alpha^2\mu_x + \frac{\alpha^2}{\eta_x})(\beta^2\mu_y + \frac{\beta^2}{\eta_y})}}\right)\log \frac{1}{\gamma}\right). 
\end{equation}
Now we are ready to define constants $\alpha, \beta$. For case $\eta_x > \eta_y$ we define $\alpha^2 = \sqrt{\frac{\eta_x}{\eta_y}}$, $\beta = 1$. For another case ($\eta_x \leq \eta_y$) we define $\alpha = 1$, $\beta^2 = \sqrt{\frac{\eta_y}{\eta_x}}$. We provide proof only for case $\eta_x > \eta_y$ due to case  $\eta_x \leq \eta_y$ is symmetric.
\begin{align*}
    T = \mathcal{O}\left(\left(\frac{\sqrt{\frac{\eta_x}{\eta_y}} L_R + \frac{2}{\eta_y}}{\sqrt{\left(\sqrt{\frac{\eta_x}{\eta_y}}\mu_x + \frac{1}{\eta_y}\right)\left(\mu_y + \frac{1}{\eta_y}\right)}}\right)\log \frac{1}{\gamma}\right) &\leq \mathcal{O}\left(\left(\frac{\sqrt{\frac{\eta_x}{\eta_y}} L_R + \frac{2}{\eta_y}}{ \frac{1}{\eta_y}}\right)\log \frac{1}{\gamma}\right)
    \\&= \mathcal{O}\left(\left(\sqrt{\eta_x\eta_y} L_R + 1\right)\log \frac{1}{\gamma}\right). 
\end{align*}

\end{document}